\documentclass[11pt,reqno]{amsart}
\usepackage[left=2cm, right=2cm, top=2cm, bottom=2cm]{geometry}

\usepackage[skip=5pt plus1pt, indent=20pt]{parskip}

\usepackage[us,12hr]{datetime}

\usepackage[T1]{fontenc}
\usepackage{mathpazo}
\usepackage{amsmath,amssymb,amscd,mathrsfs}
\usepackage{stmaryrd}
\usepackage{enumerate}
\usepackage{comment}
\usepackage{tikz-cd}

\usepackage{xcolor} 
\colorlet{mdtRed}{red!50!black}
\colorlet{dblue}{blue!50!black}
\usepackage[colorlinks,pagebackref=true]{hyperref}
\hypersetup{linkcolor=dblue,citecolor=dblue,filecolor=dullmagenta,urlcolor=mdtRed}
\renewcommand*{\backref}[1]{}
\renewcommand*{\backrefalt}[4]{[{%
		\ifcase #1 Not cited.%
		\or $\uparrow$~#2.%
		\else $\uparrow$~#2.%
		\fi%
	}]}
\usepackage[all]{xy}
\usepackage{tikz,tikz-cd,tkz-graph,enumerate}
\usetikzlibrary{matrix,arrows,decorations.pathmorphing}

\patchcmd{\subsubsection}{\itshape}{\bfseries}{}{%
	\PackageWarning{manuscript}{Subsubsection heading patch failed}%
}

\DeclareMathOperator{\Spec}{Spec}

\DeclareMathOperator{\Hom}{Hom}

\DeclareMathOperator{\Rep}{\mathbf{Rep}}

\DeclareMathOperator{\ad}{ad}
\DeclareMathOperator{\Sym}{Sym}

\DeclareMathOperator{\Quot}{Quot}
\DeclareMathOperator{\Grass}{Grass}

\DeclareMathOperator{\GL}{GL}

\DeclareMathOperator{\Dol}{Dol}
\DeclareMathOperator{\dR}{dR}

\DeclareMathOperator{\rk}{rk}

\newcommand{\mc}[1]{\mathcal{#1}}

\providecommand{\sslash}{/\!\!/}

\providecommand{\doi}[1]{%
	\href{https://doi.org/#1}{\nolinkurl{doi:#1}}%
}

\numberwithin{equation}{subsection}

\newtheorem{theorem}[equation]{Theorem}
\newtheorem{corollary}[equation]{Corollary}
\newtheorem{lemma}[equation]{Lemma}
\newtheorem{proposition}[equation]{Proposition}

\theoremstyle{definition}
\newtheorem{definition}[equation]{Definition}
\newtheorem{remark}[equation]{Remark}
\newtheorem{example}[equation]{Example}

\newtheorem{question}[equation]{Question}

\newtheorem{note}[equation]{Note}

\theoremstyle{plain}
\newtheorem{customthm}{Theorem}

\title[On Moduli of Lie algebroid connections]
{Moduli of semistable Lie algebroid connections
		and twisted Higgs bundles}

\author[S. Ghosh]{Samit Ghosh}
\address{Department of Mathematics and Statistics,
	Indian Institute of Science Education and Research Kolkata,
	Mohanpur 741246, Nadia, West Bengal, India}
\email{sg23rs005@iiserkol.ac.in}
\thanks{Corresponding author: Samit Ghosh.}

\author[A. Paul]{Arjun Paul}
\address{Department of Mathematics and Statistics,
	Indian Institute of Science Education and Research Kolkata,
	Mohanpur 741246, Nadia, West Bengal, India}
\email{arjun.paul@iiserkol.ac.in}

\subjclass[2020]{14H60, 53B15, 70G45, 17B45}

\keywords{Algebraic Lie algebroids, principal bundles,
	Lie algebroid connections, twisted Higgs bundles,
	moduli spaces, Tannakian categories}

\date{Last updated on \today\ at \currenttime\ (IST)}

\begin{document}

	\begin{abstract}
		We study integrable Lie algebroid connections and twisted Higgs
		fields on vector bundles and principal bundles over smooth
		projective families of complex curves. Under suitable hypotheses,
		we construct quasi-projective moduli spaces of semistable principal
		bundles equipped with these structures and prove that their formation
		commutes with arbitrary base change. We also characterize
		algebraically integrable transitive Lie algebroids on smooth
		irreducible complex varieties as Atiyah algebroids of principal
		bundles whose gauge groupoids are representable by separated
		schemes of finite type. Further results concern the local freeness
		of twisted Higgs sheaves and neutral Tannakian descriptions
		of the associated categories.
	\end{abstract}

	\maketitle
	
	
\section{Introduction and Main results}\label{sec:intro}
Simpson's theory of $\Lambda$--modules provides a common algebraic
framework for studying vector bundles with integrable connections,
Higgs fields, and logarithmic or meromorphic connections
\cite{Simpson-92,Simpson-94,Simpson-95,Simpson-97}.
For a suitable sheaf $\Lambda$ of rings of differential operators
on a smooth projective variety $X$, this theory constructs
quasi-projective moduli spaces of semistable $\Lambda$--modules.
The choices $\Lambda=\mathscr D_X$, the sheaf of differential
operators on $X$ \cite{Bernstein-1984}, and
$\Lambda=\Sym^\bullet TX$ give the usual de Rham and Dolbeault
moduli spaces, respectively. Over $\mathbb C$, suitable vanishing conditions on Chern classes relate these spaces through the non-abelian Hodge correspondence
\cite{Hitchin-1987,Donaldson-1987,Corlette-1988,Simpson-92}.
The Rees construction connects them through the Hodge moduli space,
which carries a natural $\mathbf G_m$--action.

Algebraic Lie algebroids extend this framework. 
For a smooth morphism $f:X\to S$ of $\mathbb C$--schemes,
a relative algebraic Lie algebroid
$\mathscr L=(L,[\, ,\,]_L,\delta)$ consists of a vector bundle $L$,
an $f^{-1}\mathscr O_S$--linear Lie bracket, and an
$\mathscr O_X$--linear, bracket-preserving anchor
$\delta:L\to T_{X/S}$ satisfying $[s,at]_L=a[s,t]_L+\delta(s)(a)t.$
Its relation to almost polynomial sheaves of rings of differential
operators is developed in \cite[Theorem~2.11]{Simpson-94}
and \cite[\S~2.2]{A.Langer-2014}, with refinements 
in characteristic zero in \cite{Tortella-2011,Tortella-2012};
see also \cite{Beilinson-Bernstein-1993} for theory of algebraic Lie algebroids.
This relation brings integrable $\mathscr L$--connections and
$\mathscr L$--twisted Higgs fields within Simpson's moduli theory;
see \S~\ref{subsec:moduli-Lambda-modules}.
The corresponding notions for principal bundles are given in Definition 
\ref{def:L-connection-principal} and \ref{def: L-twisted prncipl Higgs bundles}. Several familiar theories arise in this way. The tangent algebroid $TX$ gives ordinary integrable connections and, on the Higgs side, ordinary Higgs bundles.
For a simple normal crossings divisor $D$, the logarithmic tangent
algebroid $TX(-\log D)$ gives integrable logarithmic connections
\cite{Deligne-1970,Nitsure-1993}.
If $L$ is a line bundle with zero bracket and anchor, then
$\mathscr L$--connections are precisely
$L^{-1}$--twisted Higgs fields \cite{Nitsure-1991}.

Lie algebroid connections on vector bundles have been studied
analytically by K\v{r}i\v{z}ka \cite{Libor-2009}
and algebraically in positive characteristic by Langer
\cite{A.Langer-2014,A.Langer-2021}.
For rank-one Lie algebroids on a smooth projective curve of genus
$g\geqslant2$, Alfaya and Oliveira establish semiprojectivity
and motivic results under the assumptions
$\deg L<2-2g$ and coprime rank and degree
\cite{Alfaya-Oliveire-2024}.
Gallego, Garc\'ia-Prada, and Narasimhan study the Hitchin map,
spectral correspondence, and Hitchin--Kobayashi correspondence
for Higgs bundles twisted by a vector bundle
\cite{Gallego-Prada-Narasimhan-2024}. For principal bundles, the theory is less developed. In \cite{Ghosh-Paul-2025}, we define Lie algebroid connections on principal $G$--bundles over smooth complex projective varieties and give a sufficient condition for their existence; an equivariant
version is studied by Biswas \cite{Biswas-2026}.
For $\mathscr L=T_{X/S}$, Simpson constructs moduli spaces of
principal $G$--bundles with integrable relative connections and
of principal $G$--Higgs bundles \cite[\S\S~9--10]{Simpson-95}.

In this paper, we develop the moduli theory of integrable
Lie algebroid connections and twisted Higgs fields on principal
bundles over smooth projective families of curves.
Our main goal is to construct quasi-projective moduli spaces
under suitable hypotheses on the Lie algebroid.
The construction combines tensor product theorems for semistable
vector-bundle objects with a tensor-functor description of
principal objects. Although the base varieties are curves, the Lie algebroid
structure introduces features of the higher-dimensional theory.
In this sense, the framework lies between the classical theory
on curves and its higher-dimensional counterparts. We also characterize algebraically integrable transitive
Lie algebroids. Along the way, we establish complementary results and give
counterexamples that may be of independent interest.

\subsection*{Main results}
Unless stated otherwise, $f:X\to S$ is a smooth projective family
of geometrically connected curves over a $\mathbb C$--scheme,
$G$ is a connected reductive algebraic group over $\mathbb C$,
and $\mathscr L=(L,[\, ,\,]_L,\delta)$ is a relative algebraic
Lie algebroid with $L$ locally free.
Additional hypotheses are stated in each result.

\subsection{Algebraic integrability of transitive Lie algebroids}
We first characterize transitive algebraic Lie algebroids
that admit an algebraic groupoid integration.
Rank-one Lie algebroids on curves of genus at least two are
described in \cite[Proposition~2.1]{Gualtieri-Li-Pym-2018}
and \cite[Theorem~1.2]{Alfaya-Oliveire-2024}.
In the smooth category, integrability is governed by monodromy
obstructions
\cite{Almeida-Molino-1985,Crainic-Fernandes-2009}.
Our result gives an algebraic characterization in terms of
Atiyah algebroids and gauge groupoids.

\begin{customthm}[Theorem  \ref{thm:classification of transitive Lie algebroid}]
	\label{thm:A}
	Let $\mathcal A$ be a transitive algebraic Lie algebroid
	on a smooth irreducible variety $X$ over $\mathbb C$.
	Then $\mathcal A$ is algebraically integrable if and only if
	there exist an algebraic group $G$ and a principal $G$--bundle
	$E_G\to X$ such that $\mathcal A\cong\operatorname{At}(E_G)$
	and $(E_G\times E_G)/G$ is represented by a separated scheme
	of finite type over $\mathbb C$. More precisely, if $\mathcal G\rightrightarrows X$ integrates $\mathcal A$, then, for any closed point $p\in X$,
	one may take
	\[
	E_G=s^{-1}(p),
	\qquad
	G=s^{-1}(p)\cap t^{-1}(p).
	\]
	For these choices,
	$\mathcal G\cong\mathfrak G(E_G)$
	as algebraic groupoids.
\end{customthm}

The proof is a scheme-theoretic version of Mackenzie's description
of locally trivial Lie groupoids \cite{Mackenzie-2005}.
Right translation relates the differential of the target map
to the anchor. Transitivity and irreducibility then give a single
orbit, and the isotropy group makes each source fibre a principal
bundle over $X$.
The argument does not require projectivity, GAGA, or the
holomorphic integrability theorem of
\cite{Gengous-S-Xu-2009}; see Remark \ref{rem:no-GAGA}.

\subsection{Vector-bundle moduli and tensor products}
We construct moduli spaces of semistable integrable
$\mathscr L$--connections in Theorem  \ref{thm:moduli-vb-dR}
and relative Dolbeault and framed moduli spaces in
Theorem  \ref{thm: twisted-Dolbeault-framed-moduli}.
We also study local freeness of twisted Higgs sheaves
in higher dimensions in
Proposition \ref{prop:quotient-tangent} and \ref{prop:conditional-surface}. A key step toward principal-bundle moduli is the preservation
of semistability under tensor products.
For twisted Higgs bundles, we prove this in
Theorem  \ref{thm:twisted-Higgs-tensor} using a generalized
Hitchin--Kobayashi correspondence; see
\S~\ref{subsec:Generalized Hitchin equations}.
Combining this result with the existence of a Griffiths-transverse
filtration with semistable associated graded object
\cite[Theorem~5.5]{A.Langer-2014}
gives the following theorem.

\begin{customthm}[Theorem  \ref{thm:dR-tensor-semistability}]
	\label{thm:B}
	Let $f:X\to S$ be a smooth projective family of geometrically
	connected curves over a finite-type $\mathbb C$--scheme,
	and let $\mathscr L$ be a relative Lie algebroid on $X/S$
	with locally free underlying sheaf.
	Let $(E_1,\nabla_1)$ and $(E_2,\nabla_2)$ be vector bundles
	with integrable relative $\mathscr L$--connections.
	If both have geometrically semistable fibres, then
	\[
	\left(
	E_1\otimes E_2,\,
	\nabla_1\otimes\mathrm{id}
	+\mathrm{id}\otimes\nabla_2
	\right)
	\]
	also has geometrically semistable fibres.
	Its fibrewise slope is the sum of the slopes of the two factors.
	In particular, tensor products preserve semistability of degree zero.
	The construction commutes with arbitrary base change on $S$.
\end{customthm}

\subsection{Tensor-functor descriptions of principal objects}
We extend Simpson's tensor-functor description of principal
connections and Higgs bundles \cite[\S\S~9--10]{Simpson-95}
to the Lie algebroid setting. For an $S$--scheme $T$, write
$X_T=X\times_S T$ and let $\mathscr L_T$
be the induced relative Lie algebroid.
Let $\mathcal C_{\mathrm{dR}}(T)$ denote the category of vector
bundles on $X_T$ with integrable relative
$\mathscr L_T$--connections, and let
$\mathcal C_{\mathrm{Dol}}(T)$ denote the category of
$\mathscr L_T$--twisted Higgs bundles.

\begin{customthm}[Lemma \ref{lem:tensor-reconstruction}]
	\label{thm:C}
	For every $S$--scheme $T$, the associated-bundle construction
	induces equivalences between:
	\begin{enumerate}[(i)]
		\item
		the groupoid of principal $G$--bundles on $X_T$
		with integrable relative $\mathscr L_T$--connections
		and the groupoid of exact tensor functors $\operatorname{Rep}(G)\longrightarrow
		\mathcal C_{\mathrm{dR}}(T);$
		\item
		the groupoid of principal $\mathscr L_T$--twisted
		$G$--Higgs bundles on $X_T$
		and the groupoid of exact tensor functors $
		\operatorname{Rep}(G)\longrightarrow
		\mathcal C_{\mathrm{Dol}}(T).$
	\end{enumerate}
	The underlying vector-bundle-valued functors are automatically
	fpqc-locally trivial.
	These equivalences commute with arbitrary base change $T'\to T$.
\end{customthm}

We first reconstruct the underlying principal bundle.
Tensor-compatible differential operators on its associated
bundles then recover the connection and its integrability.
The Higgs case follows by replacing the bracket and anchor
of $\mathscr L$ with zero.
We also establish neutral Tannakian descriptions under the
hypotheses of Theorems  \ref{thm:tannakian-dR} and  \ref{thm:tannakian-Dol}:
the connection case uses a surjective anchor, whereas the
semiharmonic Higgs case imposes no such restriction.

\subsection{Moduli spaces of principal objects}
The tensor-functor description allows us to pass from
vector-bundle moduli to principal-bundle moduli.

\begin{customthm}[Theorem  \ref{thm:moduli-principal}]
	\label{thm:D}
	Under the case-dependent hypotheses of
	\S\ref{subsec:principal-relative-setup},
	let $G$ be a connected reductive algebraic group over $\mathbb C$.
	For $\bullet\in\{\dR,\Dol\}$, there is a quasi-projective
	$S$--scheme
	$\mc M^\bullet_{\mathscr L}(X/S,G)$
	that universally corepresents
	$\underline{\mc M}^\bullet_{\mathscr L}(X/S,G)$.
	Its formation commutes with arbitrary base change:
	for every morphism $S'\to S$, there is a canonical isomorphism
	\[
	\mc M^\bullet_{\mathscr L}(X/S,G)\times_S S'
	\simeq
	\mc M^\bullet_{\mathscr L_{S'}}
	(X_{S'}/S',G).
	\]
	In particular, its geometric fibre at $\bar s\to S$
	is canonically
	$\mc M^\bullet_{\mathscr L_{\bar s}}(X_{\bar s},G)$.
\end{customthm}

The moduli space is obtained as a good quotient of the framed
parameter scheme by $G$.
Local choices of frames and descent give the corepresenting
property. The good moduli space property gives universality,
and compatibility of the framed construction and quotient
with base change yields the stated isomorphism.

\subsection{Organisation}

\S\ref{sec:algebroids} recalls Lie algebroids and their
de Rham complexes.
\S\ref{sec:integrability} proves Theorem  \ref{thm:A}.
\S\ref{sec:principal-L-connections} studies principal
$\mathscr L$--connections, locally free $\mathscr L$--modules,
and $\mathscr L$--Chern classes.
\S\ref{sec:Lambda-modules} develops the vector-bundle moduli
constructions and proves Theorem   \ref{thm:B}.
\S\ref{sec:tannakian} gives the Tannakian descriptions
and proves Theorem  \ref{thm:C}.
Finally, \S\ref{sec:principal-moduli} constructs the
principal-bundle moduli spaces and proves Theorem  \ref{thm:D}.

\section{Algebraic Lie algebroids and Their de Rham complexes}
\label{sec:algebroids}

Lie algebroids were introduced by Pradines
\cite{J.Pradines-1966,J.Pradines-1967} in the study of
differentiable groupoids. They provide a common framework for Lie algebras and
tangent bundles. We use their algebraic counterparts, following \cite[\S~2]{A.Langer-2014} and \cite[\S~3]{Alfaya-Oliveire-2024}. Throughout this section, let $S$ be a $\mathbb{C}$-scheme and $f\colon X\to S$ a smooth projective morphism whose geometric fibres are connected curves. Let $G$ be an affine algebraic group over $\mathbb{C}$ with Lie algebra $\mathfrak{g}$.
We identify vector bundles with their sheaves of sections and write $
T_{X/S}
:=
\mathscr{D}er_{f^{-1}\mathscr{O}_S}
(\mathscr{O}_X,\mathscr{O}_X)$ for the relative tangent sheaf.

\subsection{Preliminaries on Lie algebroids}
\begin{definition}\label{defn:Lie algebroid}
	A \emph{relative Lie algebroid on $X/S$} is a triple $
	\mathscr{L}=(L,[\cdot,\cdot]_L,\delta),$
	where $L$ is a vector bundle on $X$ equipped with an
	$f^{-1}\mathscr{O}_S$-bilinear Lie bracket and an
	$\mathscr{O}_X$-linear morphism $
	\delta\colon L\longrightarrow T_{X/S},$ 
	called the \emph{anchor}, such that
	\[
	\delta([u,v]_L)=[\delta(u),\delta(v)],
	\qquad
	[u,av]_L=a[u,v]_L+\delta(u)(a)v
	\]
	for local sections $u,v$ of $L$ and $a$ of $\mathscr{O}_X$.
	The bracket on $T_{X/S}$ is the commutator of derivations.
\end{definition}
\noindent
Equivalently, $\mathscr{L}$ is a sheaf of Lie--Rinehart
algebras over $(f^{-1}\mathscr{O}_S,\mathscr{O}_X)$ whose
underlying $\mathscr{O}_X$-module is locally free
\cite{Rinehart-1963,A.Langer-2014}.
We also call $\mathscr{L}$ an
\emph{$\mathscr{O}_S$-Lie algebroid on $X$}.
For a local section $u$ of $L$, write
$\operatorname{ad}_u(v):=[u,v]_L$. A \emph{morphism} $\varphi\colon\mathscr{L}\to\mathscr{L}'$
of relative Lie algebroids on $X/S$ is an
$\mathscr{O}_X$-linear map satisfying
\[
\varphi([u,v]_L)
=
[\varphi(u),\varphi(v)]_{L'},
\qquad
\delta'\circ\varphi=\delta.
\]
If $\delta=0$, the bracket is $\mathscr{O}_X$-bilinear,
so $L$ is a sheaf of $\mathscr{O}_X$-Lie algebras. We call $\mathscr{L}$ \emph{transitive} if $\delta$ is
surjective, and \emph{split} if $\delta$ admits a
right inverse $\gamma\colon T_{X/S}\longrightarrow\mathscr{L}$
that is a morphism of relative Lie algebroids.
In particular, a split Lie algebroid is transitive.

\begin{example}\label{ex:algebroids}
	\begin{enumerate}[(i)]
		\item The relative tangent bundle $T_{X/S}$, with its commutator bracket and identity anchor, is a split transitive Lie algebroid. 
		
		\item Let $D\subset X$ be a nonempty relative effective Cartier divisor that is \'etale over $S$. The logarithmic tangent bundle $T_{X/S}(-\log D)=T_{X/S}(-D)$ is a Lie algebroid with the induced bracket and inclusion anchor. It is not transitive. 
		
		\item Let $\pi\colon E_G\to X$ be a principal $G$-bundle. The relative tangent sequence
		\[
		0\longrightarrow E_G\times\mathfrak{g}
		\longrightarrow T_{E_G/S}
		\longrightarrow \pi^*T_{X/S}
		\longrightarrow 0
		\]
		descends to the \emph{relative Atiyah sequence}
		\begin{equation}\label{eq:atiyah}
			0\longrightarrow \operatorname{ad}(E_G)
			\longrightarrow \operatorname{At}_{X/S}(E_G)
			\xrightarrow{\rho} T_{X/S}
			\longrightarrow 0,
		\end{equation}
		where
		\[
		\operatorname{ad}(E_G):=E_G\times^G\mathfrak{g},
		\qquad
		\operatorname{At}_{X/S}(E_G):=(\pi_*T_{E_G/S})^G.
		\]
		The bracket of invariant relative vector fields makes
		$\operatorname{At}_{X/S}(E_G)$ a transitive Lie algebroid,
		called the \emph{relative Atiyah algebroid}
		\cite{Atiyah-1957}. An $\mathscr{O}_X$-linear splitting of \eqref{eq:atiyah} is a relative algebraic connection on $E_G$.
		Its curvature takes values in
		$\operatorname{ad}(E_G)\otimes\Omega^2_{X/S}=0$,
		so the splitting preserves brackets.
		Thus $\operatorname{At}_{X/S}(E_G)$ is split if and only
		if $E_G$ admits a relative algebraic connection. For $S=\operatorname{Spec}\mathbb{C}$, the Atiyah
		algebroid of the frame bundle of a vector bundle of
		nonzero degree is transitive but not split
		\cite{Atiyah-1957}.
	\end{enumerate}
\end{example}

\subsection{The de Rham complex of a Lie algebroid}
\label{sec:algebroid-de-rham}
Let $f\colon X\to S$ be a smooth projective family of
geometrically connected curves over $\mathbb{C}$, and let
$\mathscr{L}=(L,[\cdot,\cdot]_L,\delta)$ be a relative Lie
algebroid on $X/S$. Set
\[
\Omega_{\mathscr{L}}^k
:=
\bigwedge\nolimits_{\mathscr{O}_X}^k L^*,
\qquad
\Omega_{\mathscr{L}}^0:=\mathscr{O}_X.
\]
The \emph{de Rham complex of $\mathscr{L}$} is $\bigl(\Omega_{\mathscr{L}}^\bullet,d_{\mathscr{L}}\bigr),$
where the differential $d_{\mathscr{L}}\colon
\Omega_{\mathscr{L}}^k
\longrightarrow
\Omega_{\mathscr{L}}^{k+1}$
is defined, for local sections
$\omega\in\Omega_{\mathscr{L}}^k$ and
$\ell_1,\ldots,\ell_{k+1}\in L$, by
\begin{equation}\label{eq:lie-algebroid-differential}
	\begin{aligned}
		(d_{\mathscr{L}}\omega)(\ell_1,\ldots,\ell_{k+1})
		={}&
		\sum_{i=1}^{k+1}(-1)^{i+1}
		\delta(\ell_i)
		\bigl(
		\omega(\ell_1,\ldots,\widehat{\ell_i},\ldots,\ell_{k+1})
		\bigr)
		\\
		&+
		\sum_{1\leq i<j\leq k+1}(-1)^{i+j}
		\omega\bigl(
		[\ell_i,\ell_j]_L,
		\ell_1,\ldots,\widehat{\ell_i},\ldots,
		\widehat{\ell_j},\ldots,\ell_{k+1}
		\bigr).
	\end{aligned}
\end{equation}
Here a hat denotes omission. In degree zero, this reads $
(d_{\mathscr{L}}a)(\ell)=\delta(\ell)(a)$ for local sections $a$ of $\mathscr{O}_X$ and $\ell$ of $L$. The Lie algebroid identities imply that
$d_{\mathscr{L}}^2=0$ and that $d_{\mathscr{L}}$ is an
$f^{-1}\mathscr{O}_S$-linear graded derivation:
\[
d_{\mathscr{L}}(\alpha\wedge\beta)
=
d_{\mathscr{L}}\alpha\wedge\beta
+ (-1)^p\alpha\wedge d_{\mathscr{L}}\beta,
\qquad
\alpha\in\Omega_{\mathscr{L}}^p.
\]
Thus
$\bigl(\Omega_{\mathscr{L}}^\bullet,d_{\mathscr{L}}\bigr)$
is a sheaf of differential graded
$f^{-1}\mathscr{O}_S$-algebras
\cite[\S~2.1]{A.Langer-2014}. The anchor induces an $\mathscr{O}_X$-linear morphism $\delta^*\colon\Omega_{X/S}^1\longrightarrow L^*$
characterized by $\delta^*(d_{X/S}a)=d_{\mathscr{L}}a.$
Since $\delta$ preserves brackets, its exterior powers commute
with the differentials and define a morphism of differential
graded algebras
\begin{equation}\label{eq:anchor-de-rham-morphism}
	\bigl(\textstyle\bigwedge^\bullet\delta^*\bigr)\colon
	(\Omega_{X/S}^\bullet,d_{X/S})
	\longrightarrow
	(\Omega_{\mathscr{L}}^\bullet,d_{\mathscr{L}}).
\end{equation}
\noindent
We define the \emph{de Rham cohomology of $\mathscr{L}$}
by $H_{\mathrm{DR}}^i(\mathscr{L}/S)
:=
\mathbb{H}^i
\bigl(X,\Omega_{\mathscr{L}}^\bullet\bigr).$
Taking hypercohomology in
\eqref{eq:anchor-de-rham-morphism} gives a natural map
\[
H_{\mathrm{DR}}^i(X/S)
:=
\mathbb{H}^i(X,\Omega_{X/S}^\bullet)
\longrightarrow
H_{\mathrm{DR}}^i(\mathscr{L}/S).
\]
For $\mathscr{L}=T_{X/S}$ with its identity anchor, the
associated complex is the usual relative de Rham complex,
so $H_{\mathrm{DR}}^i(T_{X/S}/S)
=
H_{\mathrm{DR}}^i(X/S).$ When $S=\operatorname{Spec}\mathbb{C}$, we write
$H_{\mathrm{DR}}^i(\mathscr{L},\mathbb{C})$ and
$H_{\mathrm{DR}}^i(X,\mathbb{C})$, respectively.

	
\section{Integrable transitive algebraic Lie algebroids}
\label{sec:integrability}
Unlike finite-dimensional Lie algebras, Lie algebroids need not
admit integrations. In the smooth category, the obstruction is
described by the monodromy groups of Crainic and Fernandes
\cite{Crainic-Fernandes-2003}; an earlier criterion in the
transitive case is due to Almeida and Molino
\cite{Almeida-Molino-1985}. We prove that an algebraically integrable transitive Lie algebroid is the Atiyah algebroid of a principal bundle
obtained from a source fibre of an integration.
Throughout this section, $X$ is a smooth irreducible variety
over $\mathbb{C}$, not necessarily projective.
Algebraic groups need not be affine, and principal bundles
are taken in the fppf topology.

\begin{definition}\label{def:groupoid}
An \emph{algebraic Lie groupoid over $X$} is a groupoid object $\mathcal{G}\rightrightarrows X$ in separated schemes of finite type over $\mathbb{C}$ whose source and target maps $s,t\colon\mathcal{G}\to X$ are smooth. Its \emph{Lie algebroid} is $\mathcal{A}(\mathcal{G})
	:=
	e^*T_{\mathcal{G}/X,s},$ where $e\colon X\to\mathcal{G}$ is the unit section. The anchor is induced by $dt$, and the bracket by
	right-invariant vector fields tangent to the source fibres. An algebraic Lie algebroid $\mathcal{A}$ on $X$ is
	\emph{algebraically integrable} if $\mathcal{A}\cong\mathcal{A}(\mathcal{G})$
	for some algebraic Lie groupoid over $X$.
\end{definition}

\begin{example}\label{ex:gauge}
	Let $\pi\colon E_G\to X$ be a principal $G$-bundle.
	Suppose that the fppf quotient
	\[
	\mathfrak{G}(E_G):=(E_G\times E_G)/G
	\]
	by the diagonal right action is represented by a separated
	scheme of finite type over $\mathbb{C}$.
	This holds when $G$ is affine: fppf-locally on $X\times X$,
	the quotient is the product with $G$, and affine schemes
	descend effectively
	\cite[\href{https://stacks.math.columbia.edu/tag/022C}{Tag 022C} and \href{https://stacks.math.columbia.edu/tag/0245}{Tag 0245}]{Stacks-project-2026}. The structure maps are
	\[
	s([a,b])=\pi(b),\qquad
	t([a,b])=\pi(a),\qquad
	e(x)=[a,a]\quad(a\in E_{G,x}),
	\]
	\[
	[a,b]^{-1}=[b,a],
	\qquad
	[a,b]\cdot[b,c]=[a,c].
	\]
	For multiplication, representatives with a common middle entry
	are chosen fppf-locally. The source and target maps are smooth,
	as may be checked after trivializing the torsor.
	These maps define the \emph{gauge groupoid} of $E_G$, and $\mathcal{A}\bigl(\mathfrak{G}(E_G)\bigr)
	\cong\operatorname{At}(E_G).$ Thus $\operatorname{At}(E_G)$ is transitive and algebraically integrable.
\end{example}

\begin{theorem} \label{thm:classification of transitive Lie algebroid}
	Let $\mathcal{A}$ be a transitive algebraic Lie algebroid on
	a smooth irreducible variety $X$ over $\mathbb{C}$.
	Then $\mathcal{A}$ is algebraically integrable if and only if
	there are an algebraic group $G$ and a principal $G$-bundle
	$E_G\to X$ such that $\mathcal{A}\cong\operatorname{At}(E_G)$
	and $(E_G\times E_G)/G$ is represented by a separated scheme
	of finite type over $\mathbb{C}$. More precisely, if $\mathcal{G}\rightrightarrows X$ integrates $\mathcal{A}$, then for any closed point $p\in X$
	one may take
	\[
	E_G=s^{-1}(p),
	\qquad
	G=s^{-1}(p)\cap t^{-1}(p).
	\]
	For these choices, $\mathcal{G}\cong\mathfrak{G}(E_G)$
	as algebraic groupoids.
\end{theorem}

\begin{proof}
	The reverse implication follows from Example \ref{ex:gauge}.
	Conversely, let $\mathcal{G}\rightrightarrows X$ integrate
	$\mathcal{A}=(V,[\cdot,\cdot],\rho)$.
	Fix a closed point $p\in X$ and set
	\[
	E:=s^{-1}(p),\qquad
	G:=s^{-1}(p)\cap t^{-1}(p),\qquad
	\tau:=t|_E,
	\]
	with scheme-theoretic fibres and intersection. We first show that $\tau$ is smooth and surjective. The scheme $E$ is smooth over $\mathbb{C}$ because $s$ is smooth. For a closed point $h\in E$, put $y=t(h)$.
	Right multiplication gives an isomorphism 
	\[
	R_h\colon s^{-1}(y)\longrightarrow E,
	\qquad k\longmapsto kh,
	\]
	which sends $1_y$ to $h$ and commutes with the target maps.
	It therefore identifies $d\tau_h$ with the anchor
	$\rho_y\colon V_y\to T_yX$.
	Since $\rho_y$ is surjective, $\tau$ is smooth. The same argument applies to $t|_{s^{-1}(x)}$ for every closed point $x\in X$. Hence each orbit in $X(\mathbb{C})$ is the set of closed points of a nonempty open subset of $X$. Irreducibility implies that any two such subsets intersect,
	so there is only one orbit.
	Thus $\tau$ contains every closed point of $X$ in its image.
	Its image is open, and its complement must therefore be empty.
	In particular, $G=\tau^{-1}(p)$ is a smooth algebraic group.
	
	Composition defines a right $G$-action on $E$, and
	\[
	E\times G\longrightarrow E\times_XE,
	\qquad (h,g)\longmapsto(h,hg),
	\]
	is an isomorphism with inverse $
	(h,h')\longmapsto(h,h^{-1}h').$ Since $\tau$ is smooth and surjective, $E\to X$ is a principal $G$-bundle. Finally, consider
	\[
	\Phi\colon E\times E\longrightarrow\mathcal{G},
	\qquad (h,h')\longmapsto hh'^{-1}.
	\]
	This map is invariant under the diagonal right action.
	For an arrow $\gamma$ over a test scheme, choose
	fppf-locally a lift $h'\in E$ of $s(\gamma)$.
	Then $h:=\gamma h'$ lies in $E$ and
	$\Phi(h,h')=\gamma$.
	Any other lift is uniquely of the form $h'g$, with $g\in G$,
	and gives the pair $(hg,h'g)$.
	Thus $\Phi$ identifies the fppf quotient $(E\times E)/G$
	with $\mathcal{G}$. The identification respects the groupoid operations, so
	$\mathcal{G}\cong\mathfrak{G}(E)$.
	Passing to Lie algebroids gives
	\[
	\mathcal{A}
	\cong\mathcal{A}(\mathcal{G})
	\cong\mathcal{A}\bigl(\mathfrak{G}(E)\bigr)
	\cong\operatorname{At}(E).
	\]
\end{proof}

	\begin{remark}\label{rem:no-GAGA}
		The proof is entirely algebraic and requires no projectivity
		assumption. The smooth analogue of our result
		\cite[Corollary~6.4]{Moerdijk-2003}, together with the
		holomorphic integrability theorem
		\cite[Theorem~3.17]{Gengous-S-Xu-2009},
		does not by itself yield the algebraic statement:
		a holomorphic integration need not be algebraizable,
		even over a projective base.For example, let $X$ be a smooth projective curve of positive genus. The source-connected, source-simply-connected integration of $TX$ is the fundamental groupoid of $X^{\mathrm{an}}$.
		Its isotropy groups are infinite and discrete, so they cannot
		be analytifications of algebraic groups of finite type.
		Nevertheless, the pair groupoid
		$X\times X\rightrightarrows X$ algebraically integrates $TX$.
		Thus algebraizing a chosen holomorphic integration is distinct
		from finding an algebraic integration.
	\end{remark}

	
\section{Algebraic $\mathscr{L}$-connections on principal bundles}
\label{sec:principal-L-connections}

We recall Lie algebroid connections on principal bundles,
following \cite[\S~2.2]{Ghosh-Paul-2025}. 
Let $f\colon X\to S$ be a smooth projective family of
geometrically connected curves over $\mathbb{C}$, and let
$\mathscr{L}=(L,[\cdot,\cdot]_L,\delta)$ be a relative Lie
algebroid on $X/S$. Let $E_G\to X$ be a principal bundle
under a linear algebraic group $G$. The adjoint
representation of $G$ on $\mathfrak{g}$ defines the
\emph{adjoint bundle} $
\operatorname{ad}(E_G):=E_G\times^G\mathfrak{g}.$ Pulling back the relative Atiyah sequence
\eqref{eq:atiyah} along the anchor
$\delta\colon L\to T_{X/S}$ gives
\begin{equation}\label{eq:atiyah-delta}
	0\longrightarrow\operatorname{ad}(E_G)
	\longrightarrow\operatorname{At}_{\delta}(E_G)
	\xrightarrow{\rho_{\delta}}L
	\longrightarrow0,
\end{equation}
where
\[
\operatorname{At}_{\delta}(E_G)
:=
\operatorname{At}_{X/S}(E_G)\times_{T_{X/S}}L,
\qquad
\rho_{\delta}(\xi,\ell):=\ell.
\]
We call \eqref{eq:atiyah-delta} the
\emph{$\mathscr{L}$-Atiyah sequence} of $E_G$.

\begin{definition}\label{def:L-connection-principal}
	For $X/S$, $\mathscr{L}$, and $E_G$ as above, an algebraic
	\emph{$\mathscr{L}$-connection} on $E_G$ is an
	$\mathscr{O}_X$-linear splitting $\nabla_{\mathscr{L}}\colon
	L\longrightarrow\operatorname{At}_{\delta}(E_G)$
	of \eqref{eq:atiyah-delta}.
\end{definition}

\begin{remark}\label{rem:pullback-Atiyah-algebroid}
	The bundle $\operatorname{At}_{\delta}(E_G)$ is a relative
	Lie algebroid with componentwise bracket $
	[(\xi,u),(\eta,v)]
	=
	\bigl([\xi,\eta],[u,v]_L\bigr)$
	and anchor $
	a_{\delta}(\xi,u)=\rho(\xi)=\delta(u),$
	where $\rho$ is the Atiyah anchor.
	Thus \eqref{eq:atiyah-delta} is an exact sequence of relative
	Lie algebroids; the kernel $\operatorname{ad}(E_G)$ has its
	fibrewise bracket and zero anchor.
\end{remark}

The extension class
\begin{equation}\label{defn:V-valued-Atiyah-class-of-E_G}
	\Phi_{\mathscr{L}}(E_G)
	\in
	\operatorname{Ext}_{\mathscr{O}_X}^1
	\bigl(L,\operatorname{ad}(E_G)\bigr)
	\cong
	H^1\bigl(
	X,\operatorname{ad}(E_G)\otimes\Omega_{\mathscr{L}}^1
	\bigr)
\end{equation}
is the \emph{$\mathscr{L}$-Atiyah class} of $E_G$.
It vanishes if and only if $E_G$ admits an
$\mathscr{L}$-connection. The \emph{curvature} of $\nabla_{\mathscr{L}}$ is
\[
\kappa_{\nabla_{\mathscr{L}}}(u,v)
:=
[\nabla_{\mathscr{L}}(u),\nabla_{\mathscr{L}}(v)]
-\nabla_{\mathscr{L}}([u,v]_L),
\]
where $u,v$ are local sections of $L$.
Since $\rho_{\delta}$ preserves brackets, this expression
takes values in $\operatorname{ad}(E_G)$.
It is alternating and $\mathscr{O}_X$-bilinear, hence defines
a section
\[
\kappa_{\nabla_{\mathscr{L}}}
\in
H^0\bigl(
X,\operatorname{ad}(E_G)\otimes\Omega_{\mathscr{L}}^2
\bigr).
\]

\begin{definition}\label{def:integrable-L-connection}
	An $\mathscr{L}$-connection is \emph{integrable}, or
	\emph{flat}, if its curvature vanishes.
	Equivalently, its splitting of \eqref{eq:atiyah-delta}
	is a morphism of relative Lie algebroids.
\end{definition}

\begin{remark}\label{rem: Def2 of L connection}
	An $\mathscr{L}$-connection on $E_G$ is equivalently an
	$\mathscr{O}_X$-linear lift of the anchor
	\[
	\widetilde{\nabla}\colon
	L\longrightarrow\operatorname{At}_{X/S}(E_G),
	\qquad
	\rho\circ\widetilde{\nabla}=\delta.
	\]
	The corresponding splitting is $
	\nabla_{\mathscr{L}}(u)
	=
	\bigl(\widetilde{\nabla}(u),u\bigr)$
	for local sections $u$ of $L$.
	The connection is integrable if and only if
	$\widetilde{\nabla}$ preserves brackets.
	
\end{remark}

\begin{example}\label{example: L-Connection on VB}
	For a vector bundle $\mathcal{E}$, an
	$\mathscr{L}$-connection is an
	$f^{-1}\mathscr{O}_S$-linear map $
	\nabla\colon
	\mathcal{E}\longrightarrow
	\mathcal{E}\otimes\Omega_{\mathscr{L}}^1$
	satisfying $\nabla(as)=a\nabla(s)+s\otimes d_{\mathscr{L}}a.$
	Writing $\nabla_u(s):=\langle\nabla(s),u\rangle$, its
	curvature is $R_{\nabla}(u,v)
	=
	[\nabla_u,\nabla_v]-\nabla_{[u,v]_L}.$
	If $R_{\nabla}=0$, the pair $(\mathcal{E},\nabla)$ is
	called a \emph{representation of $\mathscr{L}$}.
	These connections agree with those on the frame bundle,
	as shown in Lemma \ref{lem:frame-bundle}.
\end{example}

\begin{note}\label{note: L-conn-of coherent sheaf}
	The same definition applies to a coherent sheaf
	$\mathscr{F}$, with curvature given by the same operator
	formula. A coherent sheaf with an integrable
	$\mathscr{L}$-connection is called a coherent
	\emph{$\mathscr{L}$-module}.
\end{note}

\begin{lemma}\label{lem:frame-bundle}
	Let $\mathcal{E}$ be a vector bundle of rank $n\geq1$ on $X$.
	There is a canonical isomorphism of relative Lie algebroids $
	\operatorname{At}_{X/S}
	\bigl(\operatorname{Fr}(\mathcal{E})\bigr)
	\cong
	\operatorname{At}_{X/S}(\mathcal{E}),$
	where the right-hand side is the sheaf of relative
	first-order differential operators on $\mathcal{E}$
	with scalar symbol. It identifies $\mathscr{L}$-connections on $\mathcal{E}$
	with those on $\operatorname{Fr}(\mathcal{E})$ and preserves
	integrability.
\end{lemma}

\begin{proof}
	The action on $\mathcal{E}$ identifies invariant relative
	vector fields on $\operatorname{Fr}(\mathcal{E})$ with
	relative first-order differential operators having scalar
	symbol. This identification preserves brackets and identifies
	the Atiyah anchor with the symbol map and the kernel with
	$\operatorname{End}(\mathcal{E})$.
	It therefore identifies the pullback Atiyah sequences,
	their splittings, and their curvatures; see also
	\cite[Remark~2.2.9]{Ghosh-Paul-2025}.
\end{proof}

For $\mathscr{L}=T_{X/S}$ with identity anchor, this recovers
the usual notion of a relative algebraic connection
\cite{Atiyah-1957}. Such connections need not exist.
For existence criteria for $\mathscr{L}$-connections on
vector bundles and principal bundles, see
\cite{Bis-Kum-Sin-2024,Alf-Bis-Kum-Sin-2025,
	Biswas-2026,Ghosh-Paul-2025}. Now let $X$ be a smooth irreducible projective curve over
$\mathbb{C}$ and let $D$ be a nonzero effective divisor.
The nontransitive Lie algebroid $\mathscr{L}=TX(-D)$ satisfies $\Omega_{\mathscr{L}}^1=\Omega_X^1(D).$
Thus, in this case an $\mathscr{L}$-connection on a principal $G$-bundle
is a meromorphic connection with poles bounded by $D$.
It is automatically flat because $\mathscr{L}$ has rank one.
When $D$ is reduced, it is a logarithmic connection.
If $X$ has genus at least two and $G$ is connected reductive,
every principal $G$-bundle admits such a connection, as follows
from the next proposition.

\begin{proposition}\label{prop:intransitive-L-connections}
	Let $X$ be a smooth irreducible projective curve over
	$\mathbb{C}$ of genus $g\geq2$, and let $G$ be a connected
	reductive algebraic group. If $\mathscr{L}$ is a nontransitive Lie algebroid on $X$, then every principal $G$-bundle on $X$ admits an integrable algebraic $\mathscr{L}$-connection. In particular, this holds when $\operatorname{rk}(L)=1$ and $\deg(L)<2-2g$.
\end{proposition}

\begin{proof}
	Since the anchor $\delta\colon L\to TX$ is not surjective,
	its map on fibres vanishes at some closed point $x\in X$.
	It therefore factors as $L\xrightarrow{\overline{\delta}}TX(-x)
	\hookrightarrow TX.$
	The subsheaf $TX(-x)$ is closed under brackets.
	Since $\delta$ preserves brackets and the inclusion is
	injective, $\overline{\delta}$ is a Lie algebroid morphism. The line bundle $TX(-x)$ is stable and has degree
	$1-2g<2-2g$. By
	\cite[Theorem~4.3.3]{Ghosh-Paul-2025}, every principal
	$G$-bundle $E_G$ admits a $TX(-x)$-connection.
	This connection is flat because $TX(-x)$ has rank one.
	By Remark \ref{rem: Def2 of L connection}, it gives a Lie
	algebroid morphism $s_{-x}\colon TX(-x)\longrightarrow\operatorname{At}(E_G)$
	lifting the inclusion into $TX$.
	The composite $s_{-x}\circ\overline{\delta}$ is then an
	integrable $\mathscr{L}$-connection. Finally, a surjective morphism between line bundles is an isomorphism. Thus $\deg(L)<\deg(TX)$ excludes surjectivity of the anchor.
\end{proof}

\subsection{Locally free $\mathscr{L}$-modules}
\label{sec:locally-free-L-modules}
Let $f\colon X\to S$ be a smooth projective family of
geometrically connected curves over $\mathbb{C}$, and let
$\mathscr{L}=(L,[\cdot,\cdot]_L,\delta)$ be a relative Lie
algebroid on $X/S$. Besides the trivial module
$(\mathscr{O}_X,d_{\mathscr{L}})$, there is a natural
rank-one representation on the \emph{canonical module} $\omega_{\mathscr{L}}
:=
\det(L)\otimes\omega_{X/S}.$ Its connection is
\[
\nabla_{\xi}(\mu\otimes\nu)
=
\mathcal{L}_{\xi}\mu\otimes\nu
+ \mu\otimes\mathcal{L}_{\delta(\xi)}\nu,
\]
where $\mathcal{L}_{\delta(\xi)}$ is the relative Lie
derivative and
\[
\mathcal{L}_{\xi}
(\ell_1\wedge\cdots\wedge\ell_r)
=
\sum_{i=1}^r
\ell_1\wedge\cdots\wedge[\xi,\ell_i]_L
\wedge\cdots\wedge\ell_r,
\qquad r=\operatorname{rk}(L).
\]
The two terms together define an integrable
$\mathscr{L}$-connection
\cite{Evens-Lu-Weinstein-1999};
see also \cite[\S~2.2]{Gualtieri-Li-Pym-2018}.

\begin{example}
	\label{example: counter example of locally freeness of L-module}
	Let $D\subset X$ be a nonempty relative effective Cartier
	divisor and let $\mathscr{L}=T_{X/S}(-D)$, with inclusion
	anchor. Then $
	\omega_{\mathscr{L}}
	=
	T_{X/S}(-D)\otimes\omega_{X/S}
	\cong\mathscr{O}_X(-D).$ Under this identification, the canonical connection is
	$\nabla_{\xi}(a)=\delta(\xi)(a)$.
	Since relative vector fields vanishing along $D$ preserve
	its ideal, the sequence
	\[
	0\longrightarrow\mathscr{O}_X(-D)
	\longrightarrow\mathscr{O}_X
	\longrightarrow\mathscr{O}_D
	\longrightarrow0
	\]
	is exact in the category of $\mathscr{L}$-modules.
	Thus $\mathscr{O}_D$ is an integrable
	$\mathscr{L}$-module that is not locally free on $X$. For $S=\operatorname{Spec}\mathbb{C}$, this is
	\cite[Examples~2.11--2.12]{Gualtieri-Li-Pym-2018}.
	If $D$ has multiplicity $k$ at $p$ and $z$ is a local
	parameter, the frame $e=z^k$ of $\mathscr{O}_X(-D)$
	satisfies $
	\nabla e=e\otimes k\frac{dz}{z}.$ Hence its associated meromorphic connection is logarithmic, even when $D$ is not reduced.
\end{example}

\begin{remark}\label{rem:absolute-local-freeness}
	On a smooth complex variety, a coherent sheaf admitting
	an algebraic $\mathscr{L}$--connection for a transitive Lie algebroid is locally free,
	even without integrability. Indeed, a local splitting of
	the anchor induces an ordinary connection, which implies
	local freeness; by \cite[Proposition 1.7, \S VI]{Borel-1987}. For the integrable holomorphic case, see
	\cite[Proposition~2.15]{Gualtieri-Li-Pym-2018} or \cite[Theorem~5.1]{F.Bottacin-2015}. We give an elementary proof on geometric curve fibres and then deduce the relative statement.
\end{remark}

\begin{lemma}\label{lem: locally freeness of L-module}
	Suppose that $\mathscr{L}$ is transitive.
	Let $\mathscr{F}$ be a finitely presented
	$\mathscr{O}_X$-module, flat over $S$, admitting an
	$\mathscr{L}$-connection.
	Then $\mathscr{F}$ is locally free on $X$.
	The connection need not be integrable.
\end{lemma}

\begin{proof}
	Fix a geometric point $\bar{s}\to S$ and write
	$C:=X_{\bar{s}}$. The connection restricts to $\mathscr{F}_{\bar{s}}$,
	and the restricted anchor is surjective. Let $\mathscr{T}$ be the torsion subsheaf of $\mathscr{F}_{\bar{s}}$.
	The second fundamental form $
	\mathscr{T}\longrightarrow
	(\mathscr{F}_{\bar{s}}/\mathscr{T})
	\otimes L_{\bar{s}}^\vee$ is $\mathscr{O}_C$-linear and has torsion-free target, so it vanishes. Thus the connection preserves $\mathscr{T}$. Its annihilator ideal is consequently preserved by every
	local vector field. Indeed, if $a\mathscr{T}=0$, then
	\[
	0=\nabla_{\xi}(at)
	=\delta(\xi)(a)t+a\nabla_{\xi}(t)
	=\delta(\xi)(a)t
	\]
	for local sections $\xi$ of $L_{\bar{s}}$ and $t$ of
	$\mathscr{T}$, and the anchor is surjective. If $\mathscr{T}\neq0$, choose a closed point $p$ in its support. In the discrete valuation ring $\mathscr{O}_{C,p}$, its annihilator is $(z^m)$ for some $m\geq1$.
	Choosing $z$ to be an \'etale local parameter gives
	\[
	\frac{\partial}{\partial z}(z^m)
	=
	mz^{m-1}\notin(z^m),
	\]
	a contradiction in characteristic zero.
	Hence $\mathscr{F}_{\bar{s}}$ is torsion-free and therefore
	locally free. Local freeness descends from geometric fibres to fibres.
	Since $\mathscr{F}$ is finitely presented and flat over $S$,
	the fibrewise flatness criterion now implies that it is
	locally free on $X$; see \cite[\href{https://stacks.math.columbia.edu/tag/080Q}{Tag 080Q}]{Stacks-project-2026}.
\end{proof}


\begin{remark}\label{rem:relative-flatness-essential}
	Flatness over $S$ is essential in
	Lemma \ref{lem: locally freeness of L-module}.
	Let $X=C\times\mathbb{A}^1$ and $S=\mathbb{A}^1$, where
	$C$ is a smooth connected projective complex curve.
	For $\mathscr{L}=T_{X/S}$, relative differentiation preserves
	the ideal $(t)$, where $t$ is the coordinate on $S$.
	Thus $\mathscr{O}_X/(t)$ inherits an integrable
	$\mathscr{L}$-connection, but is neither flat over $S$
	nor locally free on $X$.
\end{remark}

\subsection{$\mathscr{L}$-Chern classes}
\label{sec:L-Chern-classes}
Let $f\colon X\to S$ be a smooth projective family of
geometrically connected curves over $\mathbb{C}$, and let
$\mathscr{L}=(L,[\cdot,\cdot]_L,\delta)$ be a relative Lie
algebroid on $X/S$. Let $E_G\to X$ be a principal bundle
under a connected reductive algebraic group $G$, with
Lie algebra $\mathfrak{g}$. Write $
I^k(G)
:=
\bigl(\operatorname{Sym}^k\mathfrak{g}^{\vee}\bigr)^G$ for the invariant homogeneous polynomials of degree $k$,
identified with invariant symmetric $k$-linear forms. By \eqref{eq:anchor-de-rham-morphism}, the anchor induces $
\delta^*\colon
H_{\mathrm{DR}}^\bullet(X/S)
\longrightarrow
H_{\mathrm{DR}}^\bullet(\mathscr{L}/S).$ For $q\in I^k(G)$, define
\[
q^{\mathscr{L}}(E_G/S)
:=
\delta^*\bigl(q(E_G/S)\bigr)
\in H_{\mathrm{DR}}^{2k}(\mathscr{L}/S),
\]
where $q(E_G/S)$ is the ordinary relative de Rham
characteristic class. These are the
\emph{$\mathscr{L}$-characteristic classes} of $E_G$.
For the corresponding compatibility in the smooth category,
see \cite[Proposition~4.3]{R.L.Fernandes2002};
for the holomorphic theory, see
\cite[\S~2--3]{Tortella-2011}. Suppose that $E_G$ admits an $\mathscr{L}$-connection $\nabla$, with curvature $F_{\nabla}
\in
H^0\bigl(
X,\operatorname{ad}(E_G)\otimes\Omega_{\mathscr{L}}^2
\bigr).$
Applying $q$ to the adjoint factors and taking exterior
products gives a form
$q(F_{\nabla},\ldots,F_{\nabla})$ of degree $2k$.
With compatible Chern--Weil normalizations, $
q^{\mathscr{L}}(E_G/S)
=
\bigl[q(F_{\nabla},\ldots,F_{\nabla})\bigr].$ The Bianchi identity makes this form closed, and transgression shows that its class is independent of
$\nabla$; compare \cite[Propositions~4.1--4.2]{R.L.Fernandes2002}.
In the relative algebraic setting, the formula follows
from the same identities applied to local connections
and their \v{C}ech--de Rham transgression cocycles:
pullback along the anchor gives the corresponding
$\mathscr{L}$-cocycles, and comparison with $\nabla$
gives the displayed representative.
In particular, an integrable connection implies $
q^{\mathscr{L}}(E_G/S)=0$, for $k\geq1$. 
For a representation $r\colon G\to\operatorname{GL}(W)$,
the \emph{$\mathscr{L}$-Chern classes} of the associated
bundle $E_G(W):=E_G\times^G W$ are $
c_k^{\mathscr{L}}(E_G(W)/S)
:=
\delta^*\bigl(c_k(E_G(W)/S)\bigr).$ When writing $c_k(E_G/S)$ or
$c_k^{\mathscr{L}}(E_G/S)$, we fix a representation and
suppress it from the notation. Thus
\begin{equation}
	\label{eq: relation o L-chern clases and usual chern classs}
	c_k^{\mathscr{L}}(E_G/S)
	=
	\delta^*\bigl(c_k(E_G/S)\bigr).
\end{equation}
For ordinary algebraic de Rham Chern classes, see
\cite[\href{https://stacks.math.columbia.edu/tag/0FW8}{Tag 0FW8}]{Stacks-project-2026}. When $\mathscr{L}=T_{X/S}$ with identity anchor, these are the ordinary relative classes. In general, $\delta^*$ need not be injective, even for transitive Lie algebroids.

\begin{remark}\label{rem:counterexample}
	Let $X$ be a smooth connected projective complex curve,
	let $M$ be a line bundle of positive degree, and set
	$\mathscr{L}:=\operatorname{At}(M)$.
	The identity of $\operatorname{At}(M)$ defines an
	integrable $\mathscr{L}$-connection on $M$, so $
	\delta^*c_1(M)=c_1^{\mathscr{L}}(M)=0,
	\ c_1(M)\neq0.$ Since $H_{\mathrm{DR}}^2(X/\mathbb{C})$ is
	one-dimensional, $\delta^*$ vanishes in degree two.
	All higher ordinary Chern classes vanish on $X$.
	Thus every vector bundle on $X$ has vanishing
	positive-degree $\mathscr{L}$-Chern classes. Moreover, the direct sum of the trivial connection on $\mathscr{O}_X$ and the tautological connection on $M$ makes $\mathcal{E}:=\mathscr{O}_X\oplus M$ an
	$\mathscr{L}$-module. Its invariant subbundle $M$ satisfies
	\[
	\mu(M)=\deg(M)>
	\frac{\deg(M)}{2}=\mu(\mathcal{E}),
	\]
	so $\mathcal{E}$ is not semistable.
	Hence vanishing $\mathscr{L}$-Chern classes imply neither
	vanishing ordinary Chern classes nor semistability,
	even when $\mathscr{L}$ is transitive.
\end{remark}

\begin{lemma}
	\label{lem:vanishng chern class lem for split Lie algebroid}
	Suppose that $\mathscr{L}$ is split transitive.
	Then the anchor induces an injective map
	\[
	\delta^*\colon
	H_{\mathrm{DR}}^\bullet(X/S)
	\longrightarrow
	H_{\mathrm{DR}}^\bullet(\mathscr{L}/S).
	\]
	Consequently, for every $q\in I^k(G)$, $
	q^{\mathscr{L}}(E_G/S)=0$ if and only if $q(E_G/S)=0.$
	The same equivalence holds for the Chern classes associated with any fixed representation of $G$.
\end{lemma}

\begin{proof}
	A Lie algebroid splitting
	$s\colon T_{X/S}\to L$ induces a morphism of complexes $$
	s^*\colon
	\Omega_{\mathscr{L}}^\bullet
	\longrightarrow
	\Omega_{X/S}^\bullet$$
	with $s^*\circ\delta^*=\operatorname{Id}$.
	The same identity holds on hypercohomology, so
	$\delta^*$ is injective. The assertions follow from the definitions.
\end{proof}


\section{$\Lambda$-modules and moduli of Lie algebroid connections}
\label{sec:Lambda-modules}
We recall the definitions and existence results for moduli
spaces of $\Lambda$-modules and their application to vector
bundles with Lie algebroid connections, following
\cite[\S~2]{Simpson-94},
\cite[\S~1]{A.Langer-2014}, and
\cite[\S~1]{A.Langer-2021}. Throughout this section, $S$ is a locally Noetherian scheme
over $\mathbb{C}$, and $f\colon X\to S$ is a smooth
projective family of geometrically connected curves.

\subsection{Sheaves of rings of differential operators}
\label{sec:sheaves-differential-operators}
We consider sheaves $\mathscr{A}$ of associative unital
rings on $X$ equipped with a unital morphism $\iota\colon\mathscr{O}_X\longrightarrow\mathscr{A}$
such that the image of the composite $
f^{-1}\mathscr{O}_S
\longrightarrow\mathscr{O}_X
\xrightarrow{\iota}\mathscr{A}$ is central. Left and right multiplication by
$\iota(\mathscr{O}_X)$ give $\mathscr{A}$ its $\mathscr{O}_X$-bimodule structure. Following \cite[\S~2]{Simpson-94} and
\cite[\S~1]{A.Langer-2014}, a
\emph{sheaf of rings of differential operators on $X$ over $S$} is such a sheaf $\Lambda$, equipped with an
increasing filtration $\Lambda_0\subseteq\Lambda_1\subseteq\Lambda_2
\subseteq\cdots$
by subsheaves of abelian groups satisfying:
\begin{enumerate}
	\item
	The filtration is exhaustive and multiplicative: $
	\Lambda=\bigcup_{i\geq0}\Lambda_i,
	\
	\Lambda_i\Lambda_j\subseteq\Lambda_{i+j}.$
	
	\item
	The image of $\iota\colon\mathscr{O}_X\to\Lambda$
	is $\Lambda_0$.
	
	\item
	For every $i\geq0$, the left and right
	$\mathscr{O}_X$-module structures on $
	\operatorname{Gr}_i(\Lambda)
	:=
	\Lambda_i/\Lambda_{i-1},
	\\
	\Lambda_{-1}:=0,$ coincide, and $\operatorname{Gr}_i(\Lambda)$ is coherent.
	
	\item The associated graded sheaf $
	\operatorname{Gr}(\Lambda)
	:=
	\bigoplus_{i\geq 0}\operatorname{Gr}_i(\Lambda)$
	is generated in degree $1$, i.e, the canonical graded morphism $
	T^\bullet_{\mathscr{O}_X}
	\bigl(\operatorname{Gr}_1(\Lambda)\bigr)
	\longrightarrow
	\operatorname{Gr}(\Lambda)$,
	from the tensor $\mathscr{O}_X$-algebra of
	$\operatorname{Gr}_1(\Lambda)$ to $\operatorname{Gr}(\Lambda)$,
	is surjective.
\end{enumerate}

\begin{example}[Relative differential operators]
	\label{example: usual differential operators}
	Let $f\colon X\to S$ be as above. The sheaf
	$\mathscr{D}_{X/S}$ of relative algebraic differential
	operators, filtered by order, is a sheaf of rings of
	differential operators on $X$ over $S$. Indeed,
	$\mathscr{D}_{X/S,0}=\mathscr{O}_X$, $
	\mathscr{D}_{X/S,i}\mathscr{D}_{X/S,j}
	\subseteq\mathscr{D}_{X/S,i+j}$,
	$[\mathscr{D}_{X/S,i},\mathscr{O}_X]
	\subseteq\mathscr{D}_{X/S,i-1},$ and the symbol map gives a canonical isomorphism $
	\operatorname{Gr}(\mathscr{D}_{X/S})
	\\ \cong
	\operatorname{Sym}_{\mathscr{O}_X}^{\bullet}(T_{X/S}).$
	Thus the graded pieces are coherent, their left and
	right $\mathscr{O}_X$-module structures coincide, and
	the associated graded algebra is generated in degree
	one; see \cite[\S~2.2]{A.Langer-2014}.
	For $S=\operatorname{Spec}\mathbb{C}$, this recovers
	$\mathscr{D}_X$; see also \cite{Bernstein-1984}.
\end{example}

\subsection{Universal enveloping algebras of Lie algebroids}
\label{sec:universal-enveloping-algebras}
We recall the correspondence between Lie algebroids and
split almost polynomial sheaves of rings of differential
operators, following \cite[\S~2.2]{A.Langer-2014}.
See also \cite[Theorem~2.11]{Simpson-94} and its refinement
in \cite[Theorem~4.4]{Tortella-2011}.
We retain the assumptions on $X/S$ from the preceding
subsection.

\begin{definition}\label{def:relative-Poisson-algebra}
	A \emph{sheaf of $\mathscr{O}_S$-Poisson algebras on $X$}
	is a commutative unital $\mathscr{O}_X$-algebra
	$\mathscr{A}$ with an $f^{-1}\mathscr{O}_S$-bilinear
	Lie bracket satisfying $\{a,bc\}=\{a,b\}c+b\{a,c\}.$
\end{definition}

Let $\Lambda$ be a sheaf of rings of differential operators
on $X$ over $S$, with
$\mathscr{O}_X\xrightarrow{\sim}\Lambda_0$.
It is \emph{almost polynomial} if its associated graded
algebra is commutative, $\operatorname{Gr}_1(\Lambda)$
is locally free, and the canonical map $
\operatorname{Sym}_{\mathscr{O}_X}^{\bullet}
\bigl(\operatorname{Gr}_1(\Lambda)\bigr)
\longrightarrow\operatorname{Gr}(\Lambda)$ is an isomorphism. Commutativity of the associated graded algebra gives $[\Lambda_i,\Lambda_j]\subseteq\Lambda_{i+j-1}.$
The commutator therefore induces a Poisson bracket of degree $-1$:
\[
\{\overline{u},\overline{v}\}
:=
\overline{[u,v]}
\in\operatorname{Gr}_{i+j-1}(\Lambda),
\qquad u\in\Lambda_i,\quad v\in\Lambda_j.
\]
The bracket of two degree-zero elements is zero.
In particular, $L_{\Lambda}:=\operatorname{Gr}_1(\Lambda)$
is a relative Lie algebroid, with the induced bracket and
anchor $
\delta_{\Lambda}(\overline{u})(a)
:=
\{\overline{u},a\}.$ The commutator also makes $\Lambda_1$ a relative Lie
algebroid, with anchor $u\mapsto[u,-]$ on
$\mathscr{O}_X$. It fits into an exact sequence
\begin{equation}\label{eq:first-order-Lambda-extension}
	0\longrightarrow\mathscr{O}_X
	\longrightarrow\Lambda_1
	\longrightarrow L_{\Lambda}
	\longrightarrow0,
\end{equation}
where the kernel has zero bracket and anchor.
We call $\Lambda$ \emph{split almost polynomial} if
this sequence admits an $\mathscr{O}_X$-linear,
bracket-preserving splitting.
This condition concerns the first-order extension,
not the anchor sequence of $L_{\Lambda}$.

Conversely, let
$\mathscr{L}=(L,[\cdot,\cdot]_L,\delta)$ be a relative
Lie algebroid.
Its \emph{universal enveloping algebra}
$\Lambda_{\mathscr{L}}=U(\mathscr{L})$ is the sheaf
of unital associative $f^{-1}\mathscr{O}_S$-algebras
generated by $\mathscr{O}_X$ and $L$, subject to their
additive and module relations and
\[
\xi a-a\xi=\delta(\xi)(a),
\qquad
\xi\eta-\eta\xi=[\xi,\eta]_L,
\]
where $a$ is a local section of $\mathscr{O}_X$ and
$\xi,\eta$ are local sections of $L$. Here the map from $\mathscr{O}_X$ is unital and multiplicative, and multiplication by $a$ on the left
of $\xi$ agrees with scalar multiplication in $L$. Assigning degree zero to $\mathscr{O}_X$ and degree one
to $L$ defines the order filtration.
Since $L$ is locally free, the
Poincar\'e--Birkhoff--Witt theorem gives $
\operatorname{Gr}(\Lambda_{\mathscr{L}})
\cong
\operatorname{Sym}_{\mathscr{O}_X}^{\bullet}(L);$
see \cite[Theorem~3.1]{Rinehart-1963} and
\cite[\S~2.2]{A.Langer-2014}.
In particular, $
\Lambda_{\mathscr{L},0}=\mathscr{O}_X,
\
\Lambda_{\mathscr{L},1}\cong\mathscr{O}_X\oplus L,
$
so $\Lambda_{\mathscr{L}}$ is split almost polynomial.

\begin{lemma}
	\label{lem: Lie algebroid is-equivalent-to-Lambda_L}
	The assignments $
	\mathscr{L}\longmapsto\Lambda_{\mathscr{L}},
	\ \
	\Lambda\longmapsto
	L_{\Lambda}=\operatorname{Gr}_1(\Lambda)$
	induce mutually inverse bijections between the isomorphism
	classes of relative Lie algebroids on $X/S$ with locally
	free underlying sheaf and those of split almost polynomial
	sheaves of rings of differential operators on $X$ over $S$.
\end{lemma}

\begin{proof}
	PBW identifies $L_{\Lambda_{\mathscr{L}}}$ with
	$\mathscr{L}$.
	Conversely, choose a bracket-preserving splitting of
	\eqref{eq:first-order-Lambda-extension}.
	The defining relations give a filtered algebra morphism $
	U(L_{\Lambda})\longrightarrow\Lambda.$
	Its associated graded map is the canonical isomorphism
	from $\operatorname{Sym}^{\bullet}(L_{\Lambda})$ to
	$\operatorname{Gr}(\Lambda)$.
	Hence the filtered map is an isomorphism.
\end{proof}

The corresponding equivalence for modules is as follows;
see \cite[\S~3.3]{A.Langer-2014} and \cite[Lemma 2.13]{Simpson-94}.

\begin{lemma}
	\label{lem: L-module equivalent to Lambda-L-module}
	There is a canonical equivalence between coherent
	$\mathscr{L}$-modules and left
	$\Lambda_{\mathscr{L}}$-modules that are coherent over
	$\mathscr{O}_X$.
\end{lemma}

\begin{proof}
	An integrable $\mathscr{L}$-connection on a coherent
	sheaf $\mathscr{F}$ defines an action of the generators by $
	a\cdot s=as,\ \ \xi\cdot s=\nabla_{\xi}s,$ 
	where $a$, $\xi$, and $s$ are local sections of
	$\mathscr{O}_X$, $L$, and $\mathscr{F}$, respectively. The Leibniz rule and integrability give precisely the
	defining relations of $\Lambda_{\mathscr{L}}$,
	so the action extends uniquely.
	Conversely, restriction of a
	$\Lambda_{\mathscr{L}}$-action to $L$ gives an integrable
	$\mathscr{L}$-connection.
	These constructions are inverse and preserve morphisms.
\end{proof}

	\begin{remark}[The Rees deformation]
		\label{rem: Deformation of diff algebra}
		Write $
		\Lambda_{\mathscr{L}}^{\mathrm{dR}}:=U(\mathscr{L}),
		\
		\Lambda_{\mathscr{L}}^{\mathrm{Dol}}
		:=\operatorname{Sym}_{\mathscr{O}_X}^{\bullet}(L)
		=U(\mathscr{L}_0),
		\
		\mathscr{L}_0:=(L,0,0).$ For the order filtration $F_{\bullet}$ on $U(\mathscr{L})$,
		the \emph{Rees algebra} is
		\[
		\mathcal{R}_F(U(\mathscr{L}))
		:=
		\bigoplus_{j\geq0}F_jU(\mathscr{L})\,t^j
		\subseteq U(\mathscr{L})[t],
		\]
		where $t$ is a central indeterminate.
		By PBW, its underlying left $\mathscr{O}_X[t]$-module
		is locally free, so it defines a flat family over
		$S\times\mathbb{A}^1_{\mathbb{C}}$. For $a\in\mathbb{C}$, set
		$\mathscr{L}_a:=(L,a[\cdot,\cdot]_L,a\delta)$.
		For local sections $\xi,\eta$ of $L$ and $b$ of
		$\mathscr{O}_X$, the generators $t\xi$ satisfy
		\[
		[t\xi,b]=t\delta(\xi)(b),
		\qquad
		[t\xi,t\eta]=t\bigl(t[\xi,\eta]_L\bigr),
		\]
		where the brackets on the left are commutators in the
		Rees algebra. Consequently,
		\[
		\mathcal{R}_F(U(\mathscr{L}))
		\big/
		(t-a)\mathcal{R}_F(U(\mathscr{L}))
		\cong U(\mathscr{L}_a).
		\]
		The fibres at $t=1$ and $t=0$ are therefore
		$\Lambda_{\mathscr{L}}^{\mathrm{dR}}$ and
		$\Lambda_{\mathscr{L}}^{\mathrm{Dol}}$, respectively;
		compare \cite[\S~3.6]{A.Langer-2014}.
		Thus the Rees construction connects the de Rham and
		Dolbeault algebras associated with $\mathscr{L}$.
		If both the bracket and anchor vanish, the family is constant.
	\end{remark}
	
\subsection{Moduli of $\Lambda$-modules and Lie algebroid connections on vector bundle} \label{subsec:moduli-Lambda-modules}
We recall the construction in arbitrary relative dimension
before applying it to curves. Let $S$ be a finite-type $\mathbb{C}$-scheme and
$f\colon X\to S$ a smooth projective morphism with
geometrically integral fibres of dimension $d\geq1$.
Fix a relative polarization $H$ and a sheaf $\Lambda$
of rings of differential operators on $X$ over $S$. A \emph{$\Lambda$-module} is a left $\Lambda$-module coherent over $\mathscr{O}_X$.
Since $\Lambda$ is generated by $\Lambda_1$, a coherent
subsheaf $F\subset E$ is a $\Lambda$-submodule if and
only if $\Lambda_1F\subset F$.

For an $S$-scheme $T$, write
$X_T:=X\times_S T$, with projection $p_X\colon X_T\to X$,
and let $\Lambda_T$ be the relative pullback algebra.
Its underlying left module is $
\Lambda_T
=
\mathscr{O}_{X_T}
\otimes_{p_X^{-1}\mathscr{O}_X}p_X^{-1}\Lambda.$ For a relative Lie algebroid $\mathscr{L}$, the defining
relations give $
(\Lambda_{\mathscr{L}})_T
\cong\Lambda_{\mathscr{L}_T}.$

\begin{definition}\label{def: stability-Lambda-modules}
	Let $s\colon\operatorname{Spec}k\to S$ be a geometric
	point. For a nonzero coherent sheaf $E$ on $X_s$, set $
	P_{H_s}(E,m)
	:=
	\chi\bigl(X_s,E(mH_s)\bigr),
	\
	p_{H_s}(E,m)
	:=
	\frac{P_{H_s}(E,m)}{a_{H_s}(E)},$
	where $a_{H_s}(E)$ is the leading coefficient of
	$P_{H_s}(E,m)$. A $\Lambda_s$-module $E$ is
	\emph{Gieseker $H_s$-semistable}
	(respectively, \emph{stable}) if its underlying sheaf
	is pure and every nonzero proper $\Lambda_s$-submodule
	$F\subset E$ satisfies $
	p_{H_s}(F,m)\leq p_{H_s}(E,m)
	\quad
	\text{(respectively, }p_{H_s}(F,m)<p_{H_s}(E,m)\text{)}$
	for $m\gg0$. For torsion-free $E$, define
	\[
	\mu_{H_s}(E)
	:=
	\frac{c_1(E)\cdot H_s^{d-1}}{\operatorname{rk}(E)}.
	\]
	It is \emph{slope $H_s$-semistable}
	(respectively, \emph{stable}) if every
	$\Lambda_s$-submodule $F\subset E$ with
	$0<\operatorname{rk}(F)<\operatorname{rk}(E)$ satisfies $
	\mu_{H_s}(F)\leq\mu_{H_s}(E)
	\quad
	\text{(respectively, }\mu_{H_s}(F)<\mu_{H_s}(E)\text{)}.$
\end{definition}
Unless otherwise stated, semistability below means
Gieseker semistability.
A semistable $\Lambda_s$-module admits a
Jordan--H\"older filtration $
0=E_0\subset E_1\subset\cdots\subset E_{\ell}=E$
whose quotients are stable with reduced Hilbert polynomial
$p_{H_s}(E)$.
The associated graded module
\[
\operatorname{gr}(E)
:=
\bigoplus_{i=1}^{\ell}E_i/E_{i-1}
\]
is independent of the filtration up to isomorphism.
It is \emph{polystable}, that is, a direct sum of stable
modules with the same reduced Hilbert polynomial.
Two semistable modules are \emph{$S$-equivalent} if
their associated graded modules are isomorphic.

\begin{remark}\label{rem: slope-versus-Gieseker-Lambda}
	For torsion-free $\Lambda_s$-modules,
	\[
	\text{slope stable}
	\Longrightarrow
	\text{Gieseker stable}
	\Longrightarrow
	\text{Gieseker semistable}
	\Longrightarrow
	\text{slope semistable}.
	\]
	On curves, the two notions of stability coincide, as do
	the corresponding notions of semistability.
	All these conditions test only $\Lambda_s$-submodules.
\end{remark}

Fix a Hilbert polynomial $\mathcal{P}$.
A \emph{family of semistable $\Lambda$-modules over $T$}
is a $\Lambda_T$-module $E$ whose underlying
$\mathscr{O}_{X_T}$-module is quasi-coherent, finitely
presented, and flat over $T$, with geometrically
semistable fibres of Hilbert polynomial $\mathcal{P}$. Define
\[
\underline{\mathcal{M}}_{\Lambda}(X/S,\mathcal{P})
\colon
(\mathbf{Sch}^{\mathrm{ft}}/S)^{\mathrm{op}}
\longrightarrow\mathbf{Sets}
\]
by assigning to $T\to S$ the set of isomorphism classes
of such families. Pullback defines its action on
morphisms.

\begin{theorem}[Simpson]
	\label{thm:moduli-Lambda-modules}
	Let $f\colon X\to S$, $H$, and $\Lambda$ be as above,
	and fix a Hilbert polynomial $\mathcal{P}$.
	There exists a quasi-projective $S$-scheme $
	\mathcal{M}_{\Lambda}(X/S,\mathcal{P})$
	and a natural transformation
	\[
	\varphi\colon
	\underline{\mathcal{M}}_{\Lambda}(X/S,\mathcal{P})
	\longrightarrow
	\operatorname{Hom}_S
	\bigl(\,\cdot\,,\mathcal{M}_{\Lambda}(X/S,\mathcal{P})\bigr)
	\]
	that universally corepresents the moduli functor. For every geometric point $\bar{s}\to S$, the geometric points of the fibre $\mathcal{M}_{\Lambda}(X/S,\mathcal{P})_{\bar{s}}$
	correspond to $S$-equivalence classes of semistable
	$\Lambda_{\bar{s}}$-modules with Hilbert polynomial
	$\mathcal{P}$. Equivalently, they correspond to
	isomorphism classes of polystable such modules. There is an open subscheme $
	\mathcal{M}_{\Lambda}^{s}(X/S,\mathcal{P})
	\subseteq
	\mathcal{M}_{\Lambda}(X/S,\mathcal{P})$
	that universally corepresents the subfunctor of
	families with geometrically stable fibres.
\end{theorem}

\begin{proof}
	This is \cite[Theorem~4.7]{Simpson-94}.
\end{proof}

Here universal corepresentability means that, for every
morphism $V\to\mathcal{M}_{\Lambda}$, the scheme $V$
corepresents $
\underline{\mathcal{M}}_{\Lambda}
\times_{h_{\mathcal{M}_{\Lambda}}}h_V,$
where $h_Y$ denotes the functor of points of a scheme $Y$.
It does not imply the existence of a universal family
on $\mathcal{M}_{\Lambda}$.

For the framed construction, assume that $\mathcal{P}$
has degree $d$ and corresponds to rank $n>0$.

\begin{definition}[{\cite[\S~4]{Simpson-94}}]
	\label{def: LF_p condition}
	Let $\xi\colon S\to X$ be a section.
	A family $\mathscr{F}$ of semistable $\Lambda$-modules
	over an $S$-scheme $T$ satisfies
	\emph{condition $\mathrm{LF}(\xi)$} if
	$\operatorname{gr}(\mathscr{F}_{\bar{t}})$ is locally
	free at $\xi_T(\bar{t})$ for every geometric point
	$\bar{t}\to T$, where $\xi_T\colon T\to X_T$ is the
	induced section.
\end{definition}

This condition is open and depends only on the fibrewise
$S$-equivalence class. It defines an open subscheme $
\mathcal{M}_{\Lambda}^{\mathrm{LF}(\xi)}
(X/S,\mathcal{P})
\subseteq
\mathcal{M}_{\Lambda}(X/S,\mathcal{P});$ see \cite[\S~4, properties~4.8.1--4.8.5]{Simpson-94}.

\begin{remark}\label{rem: on LF-condition}
	On a smooth curve, every Jordan--H\"older factor of a
	semistable positive-rank $\Lambda$-module is pure of
	dimension one, hence torsion-free and locally free.
	Its associated graded module is therefore locally free.
	Consequently, for families of smooth projective curves,
	condition $\mathrm{LF}(\xi)$ is automatic for every section.
\end{remark}

We now apply the construction to Lie algebroid connections. For $\mathscr{L}=T_{X/S}$, it recovers the de Rham
construction of \cite[Theorem~6.13]{Simpson-95}.

\begin{theorem}\label{thm:moduli-vb-dR}
	Let $f\colon X\to S$ be a smooth projective family of
	geometrically connected curves over a finite-type
	$\mathbb{C}$-scheme $S$, and let $\mathscr{L}$ be a
	relative Lie algebroid on $X/S$.
	Fix a relative polarization and a Hilbert polynomial
	$\mathcal{P}$ corresponding to rank $n>0$. There is a quasi-projective $S$-scheme $\mathcal{M}_{\mathscr{L}}^{\mathrm{dR}}
	(X/S,\mathcal{P})$ universally corepresenting families of rank-$n$ vector
	bundles with integrable relative $\mathscr{L}$-connections
	and geometrically semistable fibres of Hilbert polynomial
	$\mathcal{P}$. Its geometric points classify the corresponding
	$S$-equivalence classes.
	
	If $\xi\colon S\to X$ is a section, the functor of
	isomorphism classes of framed objects
	$(\mathcal{E},\nabla,\beta)$ is represented by a
	quasi-projective $S$-scheme $
	\mathcal{R}_{\mathscr{L}}^{\mathrm{dR}}
	(X/S,\xi,\mathcal{P}),$
	where, over an $S$-scheme $T$, the framing is an isomorphism $
	\beta\colon
	\xi_T^*\mathcal{E}\xrightarrow{\sim}
	\mathscr{O}_T^{\oplus n}.$
	The action of $\operatorname{GL}_{n,S}$ by change of
	framing admits a linearisation for which every point
	is GIT semistable, and
	\[
	\mathcal{R}_{\mathscr{L}}^{\mathrm{dR}}
	(X/S,\xi,\mathcal{P})
	\mathbin{/\mkern-6mu/}\operatorname{GL}_{n,S}
	\cong
	\mathcal{M}_{\mathscr{L}}^{\mathrm{dR}}
	(X/S,\mathcal{P}).
	\]
\end{theorem}

\begin{proof}
	By Lemma \ref{lem: L-module equivalent to Lambda-L-module},
	integrable $\mathscr{L}_T$-connections correspond to
	$\Lambda_{\mathscr{L}_T}^{\mathrm{dR}}$-modules.
	A semistable positive-rank module on a geometric curve
	fibre is torsion-free, hence locally free.
	Finite presentation and flatness over $T$ then imply
	local freeness of the family
	\cite[\href{https://stacks.math.columbia.edu/tag/080Q}{Tag 080Q}]{Stacks-project-2026}. The unframed assertion follows from
	Theorem  \ref{thm:moduli-Lambda-modules}. Condition $\mathrm{LF}(\xi)$ is automatic by Remark \ref{rem: on LF-condition}, so the framed construction
	and quotient description follow from \cite[Theorem~4.10]{Simpson-94}.
\end{proof}

\begin{remark}\label{rem:higher-dimensional-L-moduli}
	\begin{enumerate}[(i)]
		\item
		In higher dimensions, pure sheaves of full-dimensional
		support are torsion-free but need not be locally free.
		The framed construction applies over the
		$\mathrm{LF}(\xi)$ locus.If $\mathscr{L}$ is transitive, coherent modules and their Jordan--H\"older factors on geometric fibres are locally free by Remark \ref{rem:absolute-local-freeness}.
		Flat families are then locally free, and
		$\mathrm{LF}(\xi)$ is automatic.
		Thus Theorem  \ref{thm:moduli-vb-dR} extends to transitive
		relative Lie algebroids in arbitrary relative dimension.
		
		\item
		Let $X$ be a smooth connected projective complex variety.
		If the anchor admits a Lie algebroid splitting
		$TX\to L$, every coherent $\mathscr{L}$-module
		$\mathscr{F}$ inherits an ordinary integrable connection.
		Its positive-degree rational Chern classes vanish; see Lemma \ref{lem:vanishng chern class lem for split Lie algebroid}, so
		Riemann--Roch gives $
		P_H(\mathscr{F})
		=
		\operatorname{rk}(\mathscr{F})P_H(\mathscr{O}_X).$
		The same holds for every nonzero coherent
		$\mathscr{L}$-submodule, so Gieseker semistability is
		automatic. This applies fibrewise in relative families;
		a nonempty positive-rank moduli problem must therefore
		prescribe the reduced Hilbert polynomial of the
		structure sheaf.
		
		\item
		For $\mathscr{L}=TX(-D)$ on a smooth projective curve,
		where $D$ is a nonzero effective divisor, locally free
		modules are meromorphic connections with poles bounded
		by $D$. Their degrees need not vanish, and semistability
		is not automatic.
	\end{enumerate}
\end{remark}

\subsection{Generalized Higgs sheaves}
Let $S$ be a finite-type $\mathbb{C}$-scheme and
$f\colon X\to S$ a smooth projective morphism with
geometrically integral fibres. Fix a relative
polarization $H$ and a relative Lie algebroid
$\mathscr{L}=(L,[\cdot,\cdot]_L,\delta)$, with $L$
locally free. Write $\Omega^1_{\mathscr{L}}:=L^\vee.$
Our principal application concerns families of curves.

\begin{definition}\label{def: L-twisted prncipl Higgs bundles}
	Let $G$ be a complex reductive algebraic group and
	$E_G\to X$ a principal $G$-bundle.
	An \emph{$\mathscr{L}$-twisted Higgs field} on $E_G$
	is a section $
	\theta\in
	H^0\bigl(X,\ad(E_G)\otimes\Omega^1_{\mathscr{L}}\bigr)$
	satisfying $\theta\wedge\theta=0$, where the product
	combines the Lie bracket on $\ad(E_G)$ with the
	exterior product on $\bigwedge^\bullet L^\vee$.
	The pair $(E_G,\theta)$ is a
	\emph{principal $\mathscr{L}$-twisted $G$-Higgs bundle}.
	When $S=\operatorname{Spec}\mathbb{C}$, an
	\emph{$\mathscr{L}$-twisted Higgs sheaf} is a
	torsion-free coherent sheaf $E$ equipped with an
	$\mathscr{O}_X$-linear morphism $
	\theta\colon E\longrightarrow E\otimes L^\vee$
	satisfying $\theta\wedge\theta=0$.
	It is a \emph{Higgs bundle} if $E$ is locally free. A relative family is defined by the same morphism,
	with $E$ coherent and flat over $S$ and with
	torsion-free geometric fibres.
\end{definition}

Equivalently, a Higgs field on $E$ is an
$\mathscr{O}_X$-linear morphism $
\theta\colon L\otimes E\longrightarrow E$
such that $
[\theta(l_1),\theta(l_2)]=0$
for all local sections $l_1,l_2$ of $L$, where
$\theta(l_i)$ denotes the induced local endomorphism
of $E$.
A coherent subsheaf $F\subset E$ is
\emph{$\theta$-invariant} if
$\theta(L\otimes F)\subset F$.
Slope and Gieseker (semi)stability are defined as in
Definition \ref{def: stability-Lambda-modules}, testing only
$\theta$-invariant subsheaves; for families, these
conditions are imposed on geometric fibres.

\begin{remark}\label{rem: L-Higgs bundles as L-modules}
	For $\mathscr{L}=T_{X/S}$, this recovers relative
	Higgs sheaves and bundles, and hence their ordinary
	counterparts when $S=\operatorname{Spec}\mathbb{C}$. An $\mathscr{L}$-twisted Higgs field depends only on
	$L$, not on the bracket or anchor of $\mathscr{L}$.
	The commuting-operator condition identifies Higgs
	sheaves with modules over the Lie algebroid $
	\mathscr{L}_0=(L,0,0),$
	or equivalently over $
	\Lambda_{\mathscr{L}}^{\mathrm{Dol}}
	=
	\operatorname{Sym}_{\mathscr{O}_X}^{\bullet}(L),$
	subject to the coherence, flatness, and
	torsion-freeness conditions above.
	Thus a Higgs field is an integrable
	$\mathscr{L}_0$-connection, not in general an
	integrable $\mathscr{L}$-connection; see
	Note \ref{note: L-conn-of coherent sheaf}, 
		Lemma \ref{lem: L-module equivalent to Lambda-L-module}%
		and Remark \ref{rem: Deformation of diff algebra}.
\end{remark}

\begin{remark}[The relative curve case]
	Suppose that $f\colon X\to S$ is a family of smooth
	projective curves.
	Every flat family of torsion-free
	$\mathscr{L}$-twisted Higgs sheaves is locally free
	by \cite[Lemma~1.27]{Simpson-94}. On a fixed smooth projective curve, every
	slope-semistable $\mathscr{L}$-twisted Higgs bundle
	of degree zero admits a Jordan--H\"older filtration
	by saturated invariant subsheaves.
	Its successive quotients are stable $\mathscr{L}$--twisted Higgs bundles
	of degree zero and therefore have vanishing positive-degree rational Chern classes.
\end{remark}

\subsubsection{Local freeness of $\mathscr{L}$--twisted Higgs sheaf and vanishing Chern classes}
Throughout, $X$ is a smooth complex projective variety
of dimension $n\geq2$, $H$ is an ample divisor, and
$\mathscr{L}$ is a Lie algebroid with locally free
underlying sheaf $L$. All Chern-class vanishing
statements refer to rational cohomology.
The results depend only on $L$.

\begin{example}[Failure for an arbitrary twisting bundle]
	\label{ex:failure-arbitrary-twist}
	The filtration statement of
	\cite[Theorem~2]{Simpson-92} does not extend to
	arbitrary twisting bundles, even on a surface.
	Let $X=\mathbb{P}^1\times\mathbb{P}^1,
	\  H=\mathscr{O}_X(1,1),$
	and set $
	A=\mathscr{O}_X(3,2),
	\ B=\mathscr{O}_X(1,-6),
	\ L=B\otimes A^{-1}=\mathscr{O}_X(-2,-8).$
	Give $L$ the zero bracket and zero anchor.
	On $E=A\oplus B$, define
	\[
	\theta|_{L\otimes A}\colon
	L\otimes A\xrightarrow{\sim}B,
	\qquad
	\theta|_{L\otimes B}=0.
	\]
	Integrability is automatic since $L$ has rank one. At the generic point, a generator of $L$ acts as a nonzero nilpotent endomorphism whose unique invariant line is the generic fibre of $B$.
	Thus every invariant rank-one subsheaf $F\subset E$
	is contained in $B$, and
	\[
	\mu_H(F)\leq\mu_H(B)=-5<0=\mu_H(E).
	\]
	Hence $(E,\theta)$ is slope $H$-stable. Write $
	u=c_1(\mathscr{O}_X(1,0)),
	\ v=c_1(\mathscr{O}_X(0,1)).$
	Then $u^2=v^2=0$, $\int_Xuv=1$, and $
	c_1(E)=4u-4v,
	\ c_1(E)\cdot H=0,
	\
	\operatorname{ch}_2(E)
	=
	\frac{(3u+2v)^2+(u-6v)^2}{2}=0.$
	Moreover, $
	P_H(E,m)
	=(m+4)(m+3)+(m+2)(m-5)
	=2(m+1)^2
	=2P_H(\mathscr{O}_X,m).$
	Thus $E$ is locally free and satisfies both alternative
	additional hypotheses in Simpson's theorem.
	Nevertheless, it admits no filtration with
	Chern-trivial quotients, since additivity would give
	$c_1(E)=0$.
	Also, $
	\Delta(E)
	:=
	c_1(E)^2-2\operatorname{rk}(E)\operatorname{ch}_2(E)
	=-32.$
	Thus neither the filtration theorem nor the Bogomolov
	inequality holds for arbitrary twisting bundles.
\end{example}

\begin{proposition}[A quotient-tangent hypothesis]
	\label{prop:quotient-tangent}
	Suppose that there is an $\mathscr{O}_X$-linear
	surjection\footnote{The map $q$ need not preserve brackets or anchors.} $q\colon TX\twoheadrightarrow L.$
	Let $(E,\theta)$ be a torsion-free slope
	$H$-semistable $\mathscr{L}$-twisted Higgs sheaf
	satisfying $
		c_1(E)\cdot H^{n-1}=0,
		\
		\operatorname{ch}_2(E)\cdot H^{n-2}=0.$
	Assume that either $E$ is reflexive or $
	P_H(E)=\operatorname{rk}(E)P_H(\mathscr{O}_X).$
	Then $(E,\theta)$ admits a filtration by invariant
	subbundles whose successive quotients are slope
	$H$-stable Higgs bundles with vanishing rational
	Chern classes.
\end{proposition}

\begin{proof}
	Precomposing with $q$ gives an ordinary Higgs field $
	\vartheta\colon
	TX\otimes E
	\xrightarrow{q\otimes\mathrm{id}}
	L\otimes E
	\xrightarrow{\theta}E.$
	Its integrability follows from that of $\theta$.
	For every coherent subsheaf $F\subset E$,
	surjectivity of $q\otimes\mathrm{id}_F$ gives $\vartheta(TX\otimes F)=\theta(L\otimes F).$ Thus the two Higgs fields have the same invariant
	subsheaves, and $(E,\vartheta)$ is slope
	$H$-semistable. By \cite[Theorem~2]{Simpson-92}, $(E,\vartheta)$
	admits a filtration with locally free stable
	quotients having vanishing rational Chern classes.
	The filtration is $\theta$-invariant, and the same
	comparison on each quotient proves its stability
	for the induced $\mathscr{L}$-twisted Higgs field.
	Local freeness of the quotients makes the filtration
	one by subbundles.
\end{proof}
\noindent
For a torsion-free sheaf $F$ of rank $s>0$, set
\[
\Delta(F)
:=
2s\,c_2(F)-(s-1)c_1(F)^2
=
c_1(F)^2-2s\,\operatorname{ch}_2(F).
\]
We say that the \emph{Bogomolov inequality holds}
for $\mathscr{L}$-twisted Higgs sheaves with respect
to $H$ if every torsion-free slope $H$-semistable
such sheaf satisfies $\Delta(F)\cdot H^{n-2}\geq0.$
This is an additional hypothesis on $L$ and $H$. We will also use the following observation. Since $L$ is locally free, the Higgs field on a
torsion-free sheaf $E$ extends uniquely to
$E^{\vee\vee}$ and remains integrable.
Slope semistability is preserved: every invariant
subsheaf of $E^{\vee\vee}$ intersects $E$ in an
invariant subsheaf with the same rank and first
Chern class.

\begin{proposition}[The conditional surface case]
	\label{prop:conditional-surface}
	Suppose that $X$ is a surface and that the Bogomolov
	inequality holds for torsion-free slope
	$H$-semistable $\mathscr{L}$-twisted Higgs sheaves.
	Let $(E,\theta)$ be such a sheaf with $
	c_1(E)\cdot H=0,
	\ \operatorname{ch}_2(E)=0.$
	Then $(E,\theta)$ admits a filtration by invariant
	subbundles whose successive quotients are slope
	$H$-stable Higgs bundles with vanishing rational
	Chern classes.
\end{proposition}

\begin{proof}
	Identify degree-four classes with their degrees.
	Choose a saturated Jordan--H\"older filtration
	with stable torsion-free quotients $F_i$ of ranks
	$r_i$. Each $F_i$ has slope zero.
	The Bogomolov inequality and the Hodge index theorem
	give
	\[
	2r_i\operatorname{ch}_2(F_i)
	=
	c_1(F_i)^2-\Delta(F_i)
	\leq c_1(F_i)^2\leq0.
	\]
	Since
	$\sum_i\operatorname{ch}_2(F_i)
	=\operatorname{ch}_2(E)=0$,
	every $\operatorname{ch}_2(F_i)$ vanishes.
	Consequently, $c_1(F_i)^2=\Delta(F_i)=0.$
	The equality case of the Hodge index theorem gives
	$c_1(F_i)=0$ in $H^2(X,\mathbb{Q})$, and hence
	$c_2(F_i)=0$. Set $Q_i:=F_i^{\vee\vee}/F_i$.
	The extended Higgs sheaf $F_i^{\vee\vee}$ is slope
	semistable, so
	\[
	0\leq\Delta(F_i^{\vee\vee})
	=
	\Delta(F_i)-2r_i\operatorname{length}(Q_i)
	=
	-2r_i\operatorname{length}(Q_i).
	\]
	Thus $Q_i=0$. Each $F_i$ is reflexive and therefore
	locally free on the smooth surface
	\cite[\href{https://stacks.math.columbia.edu/tag/0B3N}{Tag 0B3N}]{Stacks-project-2026}.
	The filtration is consequently by subbundles.
\end{proof}

\begin{remark}[The higher-dimensional obstruction]
	\label{rem:higher-dimensional-Higgs-obstruction}
	Suppose that $n\geq3$, the Bogomolov inequality holds,
	and $(E,\theta)$ is a torsion-free slope
	$H$-semistable twisted Higgs sheaf satisfying $
	c_1(E)\cdot H^{n-1}=0,
	\
	\operatorname{ch}_2(E)\cdot H^{n-2}=0.$.
	Then $$
	0\leq\Delta(E)\cdot H^{n-2}
	=
	c_1(E)^2\cdot H^{n-2}\leq0.$$
	The equality case of the Hodge index theorem gives $$
	c_1(E)=0\quad \text{in }H^2(X,\mathbb{Q}),
	\quad
	\Delta(E)\cdot H^{n-2}=0.$$ These conditions do not imply local freeness.
	For a closed point $x\in X$, the rank-one Higgs sheaf
	$(\mathscr{I}_x,0)$ is slope stable and satisfies
	$
	c_1(E)\cdot H^{n-1}=0,
	\
	\operatorname{ch}_2(E)\cdot H^{n-2}=0.$, but is not
	locally free.
	Reflexivity excludes this example; a filtration
	theorem for reflexive sheaves still requires an
	equality-case argument.
\end{remark}

\begin{remark}[Reduction of the Hilbert-polynomial case]
	\label{rem:Higgs-Hilbert-polynomial-reduction}
	Assume the Bogomolov inequality and suppose that
	the conclusion of Proposition \ref{prop:quotient-tangent}
	holds for every reflexive slope $H$-semistable
	$\mathscr{L}$-twisted Higgs sheaf satisfying $
	c_1(E)\cdot H^{n-1}=0,
	\
	\operatorname{ch}_2(E)\cdot H^{n-2}=0.$.
	Then it also holds for every torsion-free such
	sheaf satisfying $
	P_H(E)=\operatorname{rk}(E)P_H(\mathscr{O}_X).$ Indeed, $E^{\vee\vee}$ is slope semistable.
	Put $r=\operatorname{rk}(E)$ and
	$Q=E^{\vee\vee}/E$.
	As above, $c_1(E)=0$ in rational cohomology.
	Since $\operatorname{ch}_2(Q)$ is the effective
	cycle of the codimension-two components of $Q$,
	\[
	0\leq\Delta(E^{\vee\vee})\cdot H^{n-2}
	=
	-2r\,\operatorname{ch}_2(Q)\cdot H^{n-2}
	\leq0.
	\]
	Thus $E^{\vee\vee}$ satisfies $
	c_1(E)\cdot H^{n-1}=0,
	\
	\operatorname{ch}_2(E)\cdot H^{n-2}=0.$.
	The assumed reflexive case gives vanishing rational
	Chern classes for $E^{\vee\vee}$, so Riemann--Roch
	yields $
	P_H(E^{\vee\vee})
	=
	rP_H(\mathscr{O}_X)
	=
	P_H(E).$ Hence $P_H(Q)=0$, so $Q=0$ and $E=E^{\vee\vee}$.
\end{remark}

\begin{question}
	Beyond the quotient-tangent hypothesis, which
	geometric conditions on $L$ ensure the Bogomolov
	inequality and the filtration conclusion of
	Proposition \ref{prop:quotient-tangent} for reflexive slope
	$H$-semistable twisted Higgs sheaves satisfying
	$$
	c_1(E)\cdot H^{n-1}=0,
	\
	\operatorname{ch}_2(E)\cdot H^{n-2}=0.$$?. By Remark \ref{rem:Higgs-Hilbert-polynomial-reduction},
	these conditions also give the corresponding
	statement with reflexivity replaced by
	$P_H(E)=\operatorname{rk}(E)P_H(\mathscr{O}_X)$.
\end{question}

\subsubsection{Relative moduli of $\mathscr{L}$-twisted Higgs sheaves and Dolbeault spaces}
\label{subsec: twisted-Dolbeault-moduli}
Let $S$ be a connected finite-type $\mathbb{C}$-scheme
and $f\colon X\to S$ a smooth projective family of
geometrically connected curves. Fix a relative
polarization $H$ and a relative Lie algebroid
$\mathscr{L}$ with locally free underlying sheaf $L$.
Set $\Lambda_{\mathscr{L}}^{\mathrm{Dol}}
:=
\operatorname{Sym}_{\mathscr{O}_X}^{\bullet}(L).$
For an $S$-scheme $T$, write $X_T:=X\times_S T$
and let $L_T$ be the pullback of $L$. Fix a Hilbert polynomial $\mathcal{P}$ corresponding to positive rank. Define \[
\underline{\mathcal{M}}^{\mathrm{Higgs}}
(X/S,\mathcal{P},\mathscr{L})
\colon
(\mathbf{Sch}/S)^{\mathrm{op}}
\longrightarrow\mathbf{Sets}
\]
by assigning to $T$ the set of isomorphism classes
of pairs $(\mathcal{E},\theta)$, where
$\mathcal{E}$ is quasi-coherent, finitely presented,
and flat over $T$, and $\theta\colon
\mathcal{E}\longrightarrow\mathcal{E}\otimes L_T^\vee$
is an $\mathscr{O}_{X_T}$-linear morphism satisfying
$\theta\wedge\theta=0$. The geometric fibres are required to be torsion-free
semistable twisted Higgs sheaves with Hilbert
polynomial $\mathcal{P}$. Here slope and Gieseker semistability coincide,
since the fibres are curves. By Remark \ref{rem: L-Higgs bundles as L-modules} 
and Theorem \ref{thm:moduli-Lambda-modules}, this functor is universally
corepresented by the quasi-projective $S$-scheme
\[
\mathcal{M}^{\mathrm{Higgs}}
(X/S,\mathcal{P},\mathscr{L})
:=
\mathcal{M}_{\Lambda_{\mathscr{L}}^{\mathrm{Dol}}}
(X/S,\mathcal{P}).
\]
Its geometric points over $\bar{s}\to S$ classify
$S$-equivalence classes of semistable twisted Higgs
sheaves on $X_{\bar{s}}$, or equivalently,
isomorphism classes of polystable objects.
The stable locus is open and universally
corepresents the stable subfunctor. The polynomial $
\mathcal{P}_0(m)
:=
\chi\bigl(X_s,\mathscr{O}_{X_s}(mH_s)\bigr)$ is independent of the geometric point $s\to S$, since $S$ is connected.
A torsion-free rank-$r$ sheaf on a fibre has degree
zero if and only if its Hilbert polynomial is
$r\mathcal{P}_0$. On complex fibres, this is also equivalent to
vanishing of all positive-degree rational Chern
classes.

\begin{definition}[Relative semiharmonic type]
	\label{def: semiharmonic-type}
	A relative $\mathscr{L}$-twisted Higgs bundle is
	\emph{of semiharmonic type} if its geometric fibres
	are semistable of degree zero. For $r>0$, define the \emph{relative Dolbeault moduli space} by
	\[
	\mathcal{M}_{\mathscr{L}}^{\mathrm{Dol}}(X/S,r)
	:=
	\mathcal{M}^{\mathrm{Higgs}}
	(X/S,r\mathcal{P}_0,\mathscr{L}).
	\]
\end{definition}

The term ``semiharmonic type'' is purely algebraic;
it does not assert the existence of harmonic metrics
for an arbitrary twisting bundle.

\begin{theorem}[Relative Dolbeault and framed
	moduli spaces]
	\label{thm: twisted-Dolbeault-framed-moduli}
	For every $r>0$, the quasi-projective $S$-scheme $\mathcal{M}_{\mathscr{L}}^{\mathrm{Dol}}(X/S,r)$
	universally corepresents the functor of rank-$r$
	families of semiharmonic type. If $\xi\colon S\to X$ is a section, the functor of isomorphism classes of such families equipped with
	a framing $\beta\colon
	\xi_T^*\mathcal{E}\xrightarrow{\sim}
	\mathscr{O}_T^{\oplus r}$
	is represented by a quasi-projective $S$-scheme $\mathcal{R}_{\mathscr{L}}^{\mathrm{Dol}}
	(X/S,\xi,r)$ carrying a universal framed family. The action of $\operatorname{GL}_{r,S}$ by change
	of framing admits a linearisation for which every
	point is GIT semistable, and there is a good quotient
	\[
	\mathcal{R}_{\mathscr{L}}^{\mathrm{Dol}}
	(X/S,\xi,r)
	\mathbin{/\mkern-6mu/}\operatorname{GL}_{r,S}
	\cong
	\mathcal{M}_{\mathscr{L}}^{\mathrm{Dol}}(X/S,r).
	\]
	In each geometric fibre, an orbit is closed if and
	only if the underlying twisted Higgs bundle is
	polystable.
\end{theorem}

\begin{proof}
	Every torsion-free sheaf on a smooth curve is
	locally free. Flatness and finite presentation
	therefore imply that every family in the Higgs
	moduli functor is locally free
	\cite[\href{https://stacks.math.columbia.edu/tag/080Q}{Tag 080Q}]{Stacks-project-2026}.
	Thus the sheaf and bundle moduli functors coincide. Condition $\mathrm{LF}(\xi_T)$ is automatic by
	Remark \ref{rem: on LF-condition}.
	The assertions follow from
	Theorem  \ref{thm:moduli-Lambda-modules} and
	\cite[Theorem~4.10]{Simpson-94}, applied to
	$\Lambda_{\mathscr{L}}^{\mathrm{Dol}}$.
\end{proof}

\begin{remark}\label{rem:relative-factorwise-vanishing}
	\label{rem: automatic-local-freeness-Dolbeault}
	For $\mathscr{L}=T_{X/S}$,
	Proposition \ref{prop:twisted-Higgs-HK} recovers
	Simpson's usual relative Dolbeault moduli and representation
	spaces; see \cite[\S~6, pp.~16--17]{Simpson-95}. The Jordan--H\"older factors of each geometric fibre
	of a family of semiharmonic type are stable
	degree-zero twisted Higgs bundles.
	On complex fibres, these factors have vanishing
	positive-degree rational Chern classes.
	This is a fibrewise statement and does not imply
	vanishing of the Chern classes of the total family
	on $X_T$. All these constructions depend only on $L$, not on
	the bracket or anchor of $\mathscr{L}$.
\end{remark}

\subsection{Tensor products of $\mathscr{L}$-twisted
	Higgs sheaves}
\label{sec:tensor-products}
Let $f\colon X\to S$ be a smooth projective family
of geometrically connected curves over a finite-type
$\mathbb{C}$-scheme $S$, and let $\mathscr{L}$ be a
relative Lie algebroid with locally free underlying
sheaf $L$ of arbitrary finite rank. A family of $\mathscr{L}$-twisted Higgs bundles is a vector bundle $\mathcal{E}$ on $X$ equipped
with an $\mathscr{O}_X$-linear morphism $
\varphi\colon
\mathcal{E}\longrightarrow\mathcal{E}\otimes L^\vee,
\ \
\varphi\wedge\varphi=0.$ Semistability and polystability are understood
fibrewise. On the curve fibres, slope and Gieseker
stability coincide. We prove the tensor-product theorem using a
Hitchin--Kobayashi correspondence on individual
complex fibres. No relative choice of Hermitian
metrics is needed.

\subsubsection{Generalized Hitchin equations on the fibres}\label{subsec:Generalized Hitchin equations}
Fix $s\in S(\mathbb{C})$ and write $C=X_s$,
identified with its associated compact Riemann
surface. Choose a K\"ahler form $\omega_C$ with $
\int_C\omega_C=1 $ and a Hermitian metric $h$ on $L_s$.
No stability assumption on $L_s$ is required. Let $(E,\varphi)$ be an $\mathscr{L}_s$-twisted Higgs bundle and $H$ a Hermitian metric on $E$.
In a local smooth $h^\vee$-orthonormal frame
$(u_\alpha)$ of $L_s^\vee$, write $
\varphi=\sum_\alpha\varphi_\alpha\otimes u_\alpha,$ where the $\varphi_\alpha$ are smooth local endomorphisms of $E$, and set
\[
[\varphi,\varphi^{\dagger_H}]_h
:=
\sum_\alpha
[\varphi_\alpha,\varphi_\alpha^{\dagger_H}].
\]
This defines a global Hermitian endomorphism,
independent of the orthonormal frame. Equivalently, for the associated morphism
$\Phi\colon E\otimes L_s\to E$,
\[
[\varphi,\varphi^{\dagger_H}]_h
=
\Phi\Phi^\dagger
- \operatorname{Tr}_{L_s}(\Phi^\dagger\Phi),
\]
where $\Phi^\dagger$ is the adjoint with respect
to $H\otimes h$ and $H$, and
$\operatorname{Tr}_{L_s}$ is the partial trace.

\begin{definition}
	A Hermitian metric $H$ solves the
	\emph{generalized Hitchin equation} if
	\begin{equation}\label{eq:generalized-Hitchin-fibre}
		i\Lambda_{\omega_C}F_H
		+ [\varphi,\varphi^{\dagger_H}]_h
		=
		2\pi\mu(E)\operatorname{Id}_E,
		\qquad
		\mu(E):=\frac{\deg E}{\operatorname{rk}E},
	\end{equation}
	where $F_H$ is the curvature of the Chern connection.
\end{definition}
We use the convention
\[
\deg E=\frac{i}{2\pi}\int_C\operatorname{tr}F_H.
\]
Taking traces and integrating
\eqref{eq:generalized-Hitchin-fibre} therefore
determines the constant $2\pi\mu(E)$.

\begin{proposition}[Fibrewise
	Hitchin--Kobayashi correspondence]
	\label{prop:twisted-Higgs-HK}
	An $\mathscr{L}_s$-twisted Higgs bundle $(E,\varphi)$
	is polystable if and only if it admits a Hermitian
	metric satisfying
	\eqref{eq:generalized-Hitchin-fibre}.
\end{proposition}

\begin{proof}
	Regard $\Phi\colon E\otimes L_s\to E$ as a twisted
	quiver bundle for the quiver with one vertex and
	one loop, with twisting bundle $L_s$.
	Its subobjects are precisely the
	$\varphi$-invariant coherent subsheaves.
	For parameters $
	\sigma=1,
	\ \
	\tau=2\pi\mu(E),$
	quiver polystability is therefore the polystability
	used here, and the quiver vortex equation is
	\eqref{eq:generalized-Hitchin-fibre}.
	The assertion follows from
	\cite[Theorem~3.1]{AlvarezConsul-GarciaPrada-2003}. The relation $\varphi\wedge\varphi=0$ restricts the
	objects but changes neither their subobjects nor
	the metric equation; see also
	\cite[Proposition~4.2 and Remark~4.3]
	{Gallego-Prada-Narasimhan-2024} for the rank-two case.
	GAGA identifies the algebraic and analytic
	subobjects and preserves degrees, so their
	stability and polystability conditions agree.
\end{proof}

\begin{remark}[The classical Hitchin equation]
	For $\mathscr{L}=T_{X/S}$, we have $L_s^\vee=K_C$.
	Normalize the metric on $K_C$ by $
	i\alpha\wedge\overline{\alpha}
	=
	|\alpha|_{h^\vee}^{\,2}\omega_C$
	for local $(1,0)$-forms $\alpha$, and give $T_C$
	the dual metric. If $\varphi^{*H}$ is the Hermitian adjoint
	$(0,1)$-form, then $
	[\varphi,\varphi^{\dagger_H}]_h
	=
	i\Lambda_{\omega_C}
	[\varphi,\varphi^{*H}].$
	Thus \eqref{eq:generalized-Hitchin-fibre} becomes
	\[
	i\Lambda_{\omega_C}
	\bigl(F_H+[\varphi,\varphi^{*H}]\bigr)
	=
	2\pi\mu(E)\operatorname{Id}_E,
	\]
	the classical Hitchin equation with our degree
	and volume conventions
	\cite{Hitchin-1987,Simpson-1988}.
	Integrability is automatic when $\operatorname{rk}L_s=1$,
	but must be imposed in higher rank.
\end{remark}

\subsubsection{The relative tensor-product theorem}
\begin{lemma}\label{lem:tensor-Higgs-well-defined}
	Let $(E_i,\varphi_i)$, $i=1,2$, be
	$\mathscr{L}$-twisted Higgs bundles on $X$.
	Set $E:=E_1\otimes E_2$ and $
	\varphi
	:=
	\varphi_1\otimes\operatorname{Id}_{E_2}
	+ \operatorname{Id}_{E_1}\otimes\varphi_2.$
	Then $(E,\varphi)$ is an $\mathscr{L}$-twisted
	Higgs bundle. The construction commutes with
	base change.
\end{lemma}

\begin{proof}
	Write $
	A:=\varphi_1\otimes\operatorname{Id}_{E_2},
	\ \
	B:=\operatorname{Id}_{E_1}\otimes\varphi_2.$
	Then
	\[
	\varphi\wedge\varphi
	=
	A\wedge A+A\wedge B+B\wedge A+B\wedge B.
	\]
	The first and last terms vanish by integrability.
	In a local frame $(u_\alpha)$ of $L^\vee$, write
	$A=\sum_\alpha A_\alpha\otimes u_\alpha$ and
	$B=\sum_\alpha B_\alpha\otimes u_\alpha$, where
	$A_\alpha,B_\alpha$ are local endomorphisms of $E$.
	Since they act on different tensor factors,
	$A_\alpha B_\beta=B_\beta A_\alpha$. Hence
	\[
	A\wedge B+B\wedge A
	=
	\sum_{\alpha,\beta}
	(A_\alpha B_\beta-B_\beta A_\alpha)
	\otimes u_\alpha\wedge u_\beta
	=0.
	\]
	Thus $\varphi$ is integrable.
	Compatibility with base change follows from
	the construction.
\end{proof}

\begin{theorem}[Relative tensor products]
	\label{thm:twisted-Higgs-tensor}
	Let $f\colon X\to S$ be a smooth projective family
	of geometrically connected curves over a finite-type
	$\mathbb{C}$-scheme, and let $\mathscr{L}$ be a
	relative Lie algebroid with locally free underlying
	sheaf of arbitrary finite rank. The tensor product of two relative
	$\mathscr{L}$-twisted Higgs bundles with geometrically
	semistable fibres again has geometrically semistable
	fibres. The same holds with semistability replaced
	by polystability.
\end{theorem}

\begin{proof}
	Integrability follows from
	Lemma \ref{lem:tensor-Higgs-well-defined}.
	We first work on a complex fibre $C=X_s$. Suppose that $(E_i,\varphi_i)$, $i=1,2$, are polystable. By Proposition \ref{prop:twisted-Higgs-HK}, choose
	metrics $H_i$ satisfying
	\eqref{eq:generalized-Hitchin-fibre} for the same
	$\omega_C$ and $h$.
	For $H=H_1\otimes H_2$, we have
	\[
	F_H
	=
	F_{H_1}\otimes\operatorname{Id}
	+ \operatorname{Id}\otimes F_{H_2}.
	\]
	Operators on different tensor factors commute with
	each other's adjoints, so
	\[
		[\varphi,\varphi^{\dagger_H}]_h
		=
		[\varphi_1,\varphi_1^{\dagger_{H_1}}]_h
		\otimes\operatorname{Id}\\
		+
		\operatorname{Id}\otimes
		[\varphi_2,\varphi_2^{\dagger_{H_2}}]_h.
	\]
	Adding the two metric equations gives
	\[
	i\Lambda_{\omega_C}F_H
	+ [\varphi,\varphi^{\dagger_H}]_h
	=
	2\pi\bigl(\mu(E_1)+\mu(E_2)\bigr)
	\operatorname{Id}_{E_1\otimes E_2}.
	\]
	Since
	$\mu(E_1\otimes E_2)=\mu(E_1)+\mu(E_2)$,
	Proposition \ref{prop:twisted-Higgs-HK} implies polystability. For semistable factors, choose saturated Jordan--H\"older filtrations with stable
	Higgs-bundle quotients $Q_{i,a}$.
	Refining the induced tensor-product filtration
	gives invariant subbundles with successive quotients $
	Q_{1,a}\otimes Q_{2,b}.$ These are polystable by the preceding argument,
	all of slope $\mu(E_1)+\mu(E_2)$.
	An extension of semistable Higgs bundles of the
	same slope is semistable, so the tensor product
	is semistable. Finally, let $k$ be any algebraically closed field
	of characteristic zero. Descend the curve, twisting
	bundle, and Higgs bundles to a finitely generated
	subfield $k_0\subset k$. An algebraic closure
	$\overline{k}_0\subset k$ embeds into $\mathbb{C}$.
	Semistability and stability are preserved and
	reflected by extensions of algebraically closed
	fields. The same holds for polystability, since
	splittings of Jordan--H\"older filtrations descend
	by base change for Higgs morphisms.
	The complex result therefore gives both assertions
	over $k$. Applying this to every geometric fibre
	proves the theorem.
\end{proof}

\begin{corollary}
	\label{cor:tensor-semiharmonic-relative}
	The tensor product of two families of semiharmonic
	type is again of semiharmonic type.
\end{corollary}

\begin{proof}
	By Theorem  \ref{thm:twisted-Higgs-tensor}, the tensor product
	is fibrewise semistable. Its fibrewise slope is
	the sum of the slopes of the factors, hence zero.
\end{proof}

\subsection{Tensor products of integrable Lie algebroid connections}
The tensor-product theorem for twisted Higgs bundles also yields the
corresponding semistability statement for integrable Lie algebroid
connections. The additional input is the existence of a gr-semistable
Griffiths transverse filtration.

\begin{theorem}[Relative tensor products of integrable connections]
	\label{thm:dR-tensor-semistability}
	Let \(f\colon X\to S\) be a smooth projective family of geometrically
	connected curves over a finite-type \(\mathbb C\)-scheme, and let
	\(\mathscr L=(L,[\, ,\,]_{\mathscr L},\delta)\) be a relative Lie
	algebroid on \(X/S\), with \(L\) locally free.
	Let \((E_1,\nabla_1)\) and \((E_2,\nabla_2)\) be vector bundles on \(X\)
	with integrable relative \(\mathscr L\)-connections. If both objects have geometrically semistable fibres, then $\bigl(E_1\otimes E_2,\nabla_1\otimes\mathrm{id} +\mathrm{id}\otimes\nabla_2\bigr)$
	has geometrically semistable fibres. Its fibrewise slope is the sum
	of the slopes of the two factors. In particular, tensor products
	preserve semistability of degree zero. The construction commutes with arbitrary base change on \(S\).
\end{theorem}

\begin{proof}
	The tensor connection is defined locally by
	\[
	\nabla_u(e_1\otimes e_2)
	=
	(\nabla_{1,u}e_1)\otimes e_2
	+e_1\otimes(\nabla_{2,u}e_2).
	\]
	Its curvature is the sum of the curvatures acting on the two tensor
	factors, so it is integrable. The formula also proves compatibility
	with arbitrary base change. It remains to establish semistability on each geometric fibre. Let \(\bar s\colon\operatorname{Spec}k\to S\) be a geometric point, and put \(C:=X_{\bar s}\) and
	\(\mathscr A:=\mathscr L_{\bar s}\).
	For simplicity, retain the notation \((E_i,\nabla_i)\) for the
	restricted objects. By \cite[Theorem~5.5]{A.Langer-2014}, each semistable
	\(\mathscr A\)-module admits a finite Griffiths transverse filtration
	\(N_i^\bullet E_i\) whose associated graded twisted Higgs bundle $\bigl(\operatorname{gr}_{N_i}E_i,\theta_i\bigr)$
	is semistable. The filtration has torsion-free successive quotients.
	Since \(C\) is a smooth curve, these quotients are locally free, so
	the filtration is by subbundles. Equip \(E_1\otimes E_2\) with the tensor-product filtration
	\[
	N^p(E_1\otimes E_2)
	:=
	\sum_{a+b=p}N_1^aE_1\otimes N_2^bE_2,
	\]
	where the filtrations are extended by \(E_i\) and \(0\) outside
	their finite ranges. This filtration is Griffiths transverse.
	Its associated graded Higgs bundle is canonically $$
	\bigl(\operatorname{gr}_{N_1}E_1,\theta_1\bigr)
	\otimes
	\bigl(\operatorname{gr}_{N_2}E_2,\theta_2\bigr).$$
	It is semistable by the Theorem  \ref{thm:twisted-Higgs-tensor}, and its slope is
	\(\mu(E_1)+\mu(E_2)\). Let \(K\subset E_1\otimes E_2\) be a nonzero proper invariant saturated subsheaf. Give \(K\) the induced filtration $
	N^pK:=K\cap N^p(E_1\otimes E_2).$ This filtration is Griffiths transverse, and its associated graded is an invariant coherent subsheaf of
	\(\operatorname{gr}_{N}(E_1\otimes E_2)\).
	Consequently,
	\[
	\mu(K)
	=
	\mu(\operatorname{gr}_{N}K)
	\leq
	\mu\bigl(\operatorname{gr}_{N}(E_1\otimes E_2)\bigr)
	=
	\mu(E_1)+\mu(E_2).
	\]
	Here ranks and degrees are preserved on passing to associated
	graded sheaves. It suffices to test saturated invariant subsheaves.
	Indeed, let $K\subset E$ be invariant and let
	$K^{\mathrm{sat}}$ be its saturation.
	The composite
	\[
	K^{\mathrm{sat}}
	\hookrightarrow E
	\xrightarrow{\nabla}
	E\otimes L_{\bar{s}}^\vee
	\longrightarrow
	(E/K^{\mathrm{sat}})\otimes L_{\bar{s}}^\vee
	\]
	is $\mathscr{O}_C$-linear by the Leibniz rule.
	It vanishes generically because $K$ is invariant
	and agrees generically with $K^{\mathrm{sat}}$.
	Its target is torsion-free, so it vanishes everywhere.
	Thus $K^{\mathrm{sat}}$ is invariant, and the
	displayed inequality proves semistability.
	Applying the argument to every geometric fibre
	proves the theorem.
\end{proof}

\begin{remark}
	\label{rem:dR-tensor-scope}
	The theorem establishes tensor semistability, which is the input
	needed for the degree-zero Tannakian categories and the principal
	moduli construction. No assertion about preservation of
	polystability for integrable Lie algebroid connections is needed
	in those arguments.
\end{remark}

\section{Tannakian description of $\mathscr{L}$-twisted principal objects}
\label{sec:tannakian}
Let $f\colon X\to S$ be a smooth projective family
of geometrically connected curves over a
$\mathbb{C}$-scheme $S$.
Let $G$ be a connected reductive algebraic group
over $\mathbb{C}$, and let
$\mathscr{L}=(L,[\cdot,\cdot]_L,\delta)$ be a relative
Lie algebroid on $X/S$, with $L$ locally free.
Write $\Omega^1_{\mathscr{L}}=L^\vee$. For an $S$-scheme $T$, set $
X_T:=X\times_S T, \ \ f_T\colon X_T\to T,$ and let $\mathscr{L}_T$ be the induced relative Lie algebroid. All connections below are relative to the indicated base.

\subsection{Relative tensorial reconstruction}
Let $\operatorname{Rep}(G)$ be the category of
finite-dimensional algebraic representations of
$G$ over $\mathbb{C}$. This category is semisimple.
We give $\operatorname{Vect}(X_T)$ its usual exact
tensor structure: short exact sequences are those
exact as sequences of $\mathscr{O}_{X_T}$-modules.
A morphism is \emph{strict} if its cokernel is
locally free. All tensor functors are $\mathbb{C}$-linear and
symmetric, and preserve the unit.
Morphisms between them are tensor natural
isomorphisms.

\begin{lemma}\label{lem:strict-images}
	Every additive functor $F\colon\operatorname{Rep}(G)
	\longrightarrow\operatorname{Vect}(X_T)$
	sends every morphism to a strict morphism.
\end{lemma}

\begin{proof}
	For a morphism $\varphi\colon V\to W$, put
	$I=\operatorname{im}\varphi$.
	Semisimplicity gives decompositions $
	V\cong\ker\varphi\oplus I,
	\
	W\cong I\oplus\operatorname{coker}\varphi,$ under which $\varphi$ is the identity on $I$ and zero on $\ker\varphi$.
	Since $F$ preserves these decompositions, $\operatorname{coker}F(\varphi)
	\cong F(\operatorname{coker}\varphi),$ which is locally free.
\end{proof}

A principal $G$-bundle $E_G$ on $X_T$ is a torsor
under $G\times_{\mathbb{C}}X_T$. Its associated-bundle functor is
\begin{equation}\label{eq:nori}
	\rho_{E_G}\colon
	\operatorname{Rep}(G)\longrightarrow
	\operatorname{Vect}(X_T),
	\qquad
	W\longmapsto E_G(W):=E_G\times^G W.
\end{equation}

\begin{proposition}[Tensorial reconstruction]
	\label{prop:nori}
	For every $S$-scheme $T$, the associated-bundle
	construction induces an equivalence between the
	groupoid of principal $G$-bundles on $X_T$ and
	the groupoid of exact tensor functors $
	\rho\colon\operatorname{Rep}(G)
	\longrightarrow\operatorname{Vect}(X_T).$
	Every such functor is fpqc-locally tensor-isomorphic
	to $W\longmapsto W\otimes_{\mathbb{C}}\mathscr{O}_{X_T}.$
	These equivalences commute with base change $T'\to T$.
\end{proposition}

\begin{proof}
	Put $Y=X_T$ and let $\omega_G$ be the forgetful
	functor.
	Scheme-valued Tannakian reconstruction identifies $
	E_\rho
	:=
	\underline{\operatorname{Isom}}^\otimes
	\bigl(\omega_G\otimes_{\mathbb{C}}\mathscr{O}_Y,\rho\bigr)$ with a $G_Y$-torsor, and evaluation gives a tensor
	isomorphism $\rho_{E_\rho}\cong\rho$; see
	\cite[Proposition~2.9]{Nori-1976} and
	\cite[Lemma~3.1]{Anschutz-2018}.
	In particular, $\rho$ is fpqc-locally trivial.
	Conversely, the associated-bundle functor of a
	torsor is exact and becomes the trivial functor
	after pullback to the torsor.
	These constructions are quasi-inverse, and the
	$\underline{\operatorname{Isom}}^\otimes$-construction
	commutes with base change.
\end{proof}

Let $\mathcal{C}_{\mathrm{dR}}(T)$ be the category
of vector bundles on $X_T$ with integrable relative
$\mathscr{L}_T$-connections, and let
$\mathcal{C}_{\mathrm{Dol}}(T)$ be the category of
$\mathscr{L}_T$-twisted Higgs bundles.
Morphisms preserve the additional structures, and
short exact sequences are those exact on the
underlying vector bundles.
Both categories have the usual tensor products
and duals.
Their tensor units are $\mathscr{O}_{X_T}$ with
its canonical $\mathscr{L}_T$-connection and with
the zero Higgs field, respectively.

\begin{lemma}[Reconstruction of the additional structure]
	\label{lem:tensor-reconstruction}
	For every $S$-scheme $T$, the associated-bundle
	construction induces equivalences between:
	\begin{enumerate}[(i)]
		\item
		the groupoid of principal $G$-bundles on $X_T$
		with integrable relative $\mathscr{L}_T$-connections
		and the groupoid of exact tensor functors $
		\operatorname{Rep}(G)
		\longrightarrow\mathcal{C}_{\mathrm{dR}}(T);$
		
		\item
		the groupoid of principal $\mathscr{L}_T$-twisted
		$G$-Higgs bundles on $X_T$ and the groupoid of
		exact tensor functors $
		\operatorname{Rep}(G)
		\longrightarrow\mathcal{C}_{\mathrm{Dol}}(T).$
	\end{enumerate}
	The underlying vector-bundle-valued functors are
	automatically fpqc-locally trivial.
	These equivalences commute with arbitrary base
	change $T'\to T$.
\end{lemma}

\begin{proof}
	Put $Y=X_T$.
	
	\smallskip
	\noindent
	\emph{Tensor derivations and the Atiyah algebroid.}
	For a principal $G$-bundle $E_G$ on $Y$, let
	$\mathcal{D}^\otimes_{Y/T}(\rho_{E_G})$ be the
	sheaf of natural families $(D_W)_{W\in\operatorname{Rep}(G)}$
	of $T$-relative differential operators on the
	associated bundles, with a common scalar symbol
	$v\in T_{Y/T}$, satisfying
	\[
	D_W(as)=aD_W(s)+v(a)s,
	\]
	\[
	D_{W\otimes W'}(s\otimes t)
	=
	D_W(s)\otimes t+s\otimes D_{W'}(t),
	\qquad
	D_{\mathbb{C}}=v.
	\]
	Here $a$ is a local section of $\mathscr{O}_Y$,
	and $s,t$ are local sections of the indicated
	associated bundles.
	The last identity means that the operator on the
	unit bundle is $a\mapsto v(a)$. The action on associated bundles gives an
	isomorphism of relative Lie algebroids
	\begin{equation}\label{eq:atiyah-tensor-derivations}
		\operatorname{At}_{Y/T}(E_G)
		\xrightarrow{\sim}
		\mathcal{D}^\otimes_{Y/T}(\rho_{E_G}),
	\end{equation}
	where the anchor on the right is the symbol map
	and the bracket is the commutator. To verify this, work on an affine \'etale cover $U=\operatorname{Spec}R$ trivializing $E_G$.
	Relative derivations lift uniquely along \'etale
	morphisms. Subtracting the coefficientwise action of $v$
	from each $D_W$ gives a natural family of
	$R$-linear endomorphisms $(A_W)_W$ satisfying
	\[
	A_{W\otimes W'}
	=
	A_W\otimes\operatorname{Id}
	+ \operatorname{Id}\otimes A_{W'},
	\qquad
	A_{\mathbb{C}}=0.
	\]
	Thus $(\operatorname{Id}+\varepsilon A_W)_W$ is
	a tensor automorphism over
	$R[\varepsilon]/(\varepsilon^2)$ reducing to the
	identity over $R$.
	By
	\cite[Proposition~2.8]{Deligne-Milne-Ogus-Shih-1982},
	such families correspond to
	\[
	\ker\bigl(
	G(R[\varepsilon]/(\varepsilon^2))
	\longrightarrow G(R)
	\bigr)
	\cong\mathfrak{g}\otimes_{\mathbb{C}}R.
	\]
	This identifies the two sides of
	\eqref{eq:atiyah-tensor-derivations} locally.
	The identification is canonical, hence descends,
	and preserves anchors and brackets.
	
	\smallskip
	\noindent
	\underline{\emph{proof of (i)}:}
	Let $\rho\colon\operatorname{Rep}(G)
	\longrightarrow\mathcal{C}_{\mathrm{dR}}(T)$ be an exact tensor functor.
	By Proposition \ref{prop:nori}, its underlying functor
	reconstructs a principal bundle $E_G$ and a tensor
	identification $\rho(W)\cong E_G(W)$.
	Transport the connections through this identification. For each local section $u$ of $L_T$, the operators $(\nabla^W_u)_W$ form a natural tensor-compatible family with symbol $\delta_T(u)$.
	Compatibility with the unit gives $\nabla^{\mathbb{C}}_u(a)=\delta_T(u)(a)$
	for local sections $a$ of $\mathscr{O}_Y$.
	By \eqref{eq:atiyah-tensor-derivations}, these
	families define an $\mathscr{O}_Y$-linear map $\eta\colon L_T\longrightarrow
	\operatorname{At}_{Y/T}(E_G)$ with anchor $\delta_T$.
	Hence $u\longmapsto(\eta(u),u)$ splits the pulled-back Atiyah sequence
	\[
	0\longrightarrow\ad(E_G)
	\longrightarrow\operatorname{At}_{\delta_T}(E_G)
	\xrightarrow{q}L_T
	\longrightarrow0.
	\]
	For local sections $u,v$ of $L_T$, integrability
	of the associated connections gives $[\nabla^W_u,\nabla^W_v]
	=
	\nabla^W_{[u,v]_{L_T}}$ for every representation $W$.
	By \eqref{eq:atiyah-tensor-derivations}, this implies $
	[\eta(u),\eta(v)]=\eta([u,v]_{L_T}).$
	Thus the splitting defines an integrable principal $\mathscr{L}_T$-connection.
	
	Conversely, an integrable principal connection
	acts on every associated bundle and gives an exact
	tensor functor to $\mathcal{C}_{\mathrm{dR}}(T)$.
	The constructions are quasi-inverse on objects.
	
	\smallskip
	\noindent
	\underline{\emph{proof of (ii)}:}
	Apply the connection construction to $\mathscr{L}_{T,0}:=(L_T,0,0).$
	Its locally free integrable modules are precisely
	the $\mathscr{L}_T$-twisted Higgs bundles.
	The pulled-back Atiyah algebroid is
	\[
	\operatorname{At}_0(E_G)
	\cong\ad(E_G)\oplus L_T,
	\qquad
	[(a,u),(b,v)]=([a,b],0),
	\]
	for local sections $a,b$ of $\ad(E_G)$ and
	$u,v$ of $L_T$.
	A splitting therefore has the form
	\[
	u\longmapsto(\theta_G(u),u),
	\qquad
	\theta_G\colon L_T\longrightarrow\ad(E_G),
	\]
	and preserves brackets exactly when $[\theta_G(u),\theta_G(v)]=0.$
	This is the integrability condition for a principal twisted Higgs field.
	
	\smallskip
	\noindent
	\underline{\emph{Morphisms and base change}}:
	By Proposition \ref{prop:nori}, a tensor natural isomorphism
	of the underlying functors comes from a unique
	principal-bundle isomorphism
	$\psi\colon E_G\xrightarrow{\sim}E'_G$.
	Write $\psi_W$ for its action on $E_G(W)$.
	Under \eqref{eq:atiyah-tensor-derivations}, the
	induced Atiyah-algebroid isomorphism acts by $(D_W)_W
	\longmapsto
	(\psi_WD_W\psi_W^{-1})_W.$
	Thus $\psi$ preserves the reconstructed connection
	if and only if every $\psi_W$ preserves the
	associated connection. The same argument applies to Higgs fields.
	This proves full faithfulness. Pullback of a principal $\mathscr{L}$--connection or $\mathscr{L}$--Higgs field induces the pulled-back structures on every associated bundle. Full faithfulness therefore gives canonical base-change isomorphisms, compatible with composition.
\end{proof}

\begin{remark}
	For ordinary Higgs bundles, compare
	\cite[Lemma~9.2]{Simpson-95} and
	\cite[remarks following Lemma~6.13]{Simpson-92}.
	For twisted Higgs bundles on a fixed smooth
	projective curve, see also
	\cite[Theorem~3.11]{Balaji-Parameswaran-2011}. The proof above applies to arbitrary test schemes, including nonreduced ones. In the Higgs case, it depends only on $L_T$, not on the bracket or anchor.
\end{remark}

\subsection{The de Rham side}
A \emph{neutral Tannakian category over a field $k$} is a
rigid abelian $k$-linear symmetric tensor category
$\mathcal{T}$ with $\operatorname{End}(\mathbb{1})=k$,
admitting an exact faithful tensor functor $\omega\colon\mathcal{T}\longrightarrow\operatorname{Vect}_k.$
Here $\operatorname{Vect}_k$ denotes finite-dimensional
$k$-vector spaces, and $\omega$ is called a
\emph{fibre functor}. Tannakian duality gives an equivalence $\mathcal{T}
\simeq
\operatorname{Rep}_k
\bigl(\operatorname{Aut}^{\otimes}(\omega)\bigr),$
where $\operatorname{Aut}^{\otimes}(\omega)$ is an
affine $k$-group scheme; see
\cite[Definition~2.19 and Theorem~2.11]
{Deligne-Milne-Ogus-Shih-1982}.

For every $S$-scheme $T$, the category $\mathcal{C}_{\mathrm{dR}}(T)
=
\operatorname{Vect}^{\mathrm{int}}_{\mathscr{L}_T}
(X_T/T)$
is a rigid exact tensor category, with exactness
defined on the underlying vector bundles.
The tensor and dual connections are given by $
\nabla^{E\otimes F}_u(e\otimes f)
=
\nabla^E_u(e)\otimes f
+ e\otimes\nabla^F_u(f)$
and
\[
(\nabla^{E^\vee}_u\lambda)(e)
=
\delta_T(u)\bigl(\lambda(e)\bigr)
- \lambda(\nabla^E_u e).
\]
Here $u,e,f,\lambda$ are local sections of
$L_T,E,F,E^\vee$, respectively, over a common
open subset of $X_T$.
Both constructions preserve integrability, and
the usual symmetry, evaluation, and coevaluation
maps are horizontal.
The tensor unit is $\mathbb{1}_{\mathrm{dR}}
=
(\mathscr{O}_{X_T},d_{\mathscr{L}_T}),
\
(d_{\mathscr{L}_T})_u(a)=\delta_T(u)(a),$
for local sections $a$ of $\mathscr{O}_{X_T}$.
Tensoring is exact because the underlying bundles
are locally free. Pullback along any morphism $T'\to T$ is an exact
tensor functor: short exact sequences of vector
bundles are locally split and remain exact after
pullback. The maximal subgroupoids form an fpqc stack over
$S$, since vector bundles and relative connections
admit fpqc descent, and the Leibniz and integrability
identities can be checked fpqc-locally. These categories of families need not be abelian. Under transitivity, their geometric-fibre categories
are neutral Tannakian.

\begin{theorem}\label{thm:tannakian-dR}
	Let $f\colon X\to S$ be a smooth projective family of geometrically connected curves over a $\mathbb{C}$-scheme, and let $\mathscr{L}$ be a
	relative Lie algebroid with locally free underlying
	sheaf and surjective anchor $\delta\colon L\longrightarrow T_{X/S}.$
	For every geometric point $\bar{s}\colon\operatorname{Spec}K\to S$, the category $
	\mathcal{T}_{\bar{s}}
	:=
	\operatorname{Vect}^{\mathrm{int}}_{\mathscr{L}_{\bar{s}}}
	(X_{\bar{s}}/K)$
	is neutral Tannakian over $K$.
	For every $x\in X_{\bar{s}}(K)$, evaluation
	\[
	\omega_x\colon\mathcal{T}_{\bar{s}}
	\longrightarrow\operatorname{Vect}_K,
	\qquad
	(E,\nabla)\longmapsto E_x,
	\]
	is a fibre functor.
\end{theorem}

\begin{proof}
	Put $Y=X_{\bar{s}}$ and
	$\mathscr{A}=\mathscr{L}_{\bar{s}}$.
	The anchor remains surjective because relative
	tangent bundles commute with smooth base change.
	The characteristic-zero local-freeness criterion
	of Lemma \ref{lem: locally freeness of L-module}
	therefore implies that every coherent
	$\mathscr{A}$-module is locally free. For a morphism $a\colon(E,\nabla)\longrightarrow(F,\nabla')$
	in $\mathcal{T}_{\bar{s}}$, the sheaf-theoretic
	kernel, image, and cokernel inherit integrable
	$\mathscr{A}$-connections.
	Restriction to the kernel uses local freeness
	of the underlying sheaf of $\mathscr{A}$.
	These coherent modules are therefore locally free.
	The usual coimage--image isomorphism respects the
	connections, so $\mathcal{T}_{\bar{s}}$ is abelian.
	Every morphism is strict, and exact sequences
	are exact on the underlying vector bundles. The tensor and dual constructions make
	$\mathcal{T}_{\bar{s}}$ rigid.
	Moreover,
	\[
	\operatorname{End}(\mathbb{1}_{\mathrm{dR}})
	=
	\{a\in H^0(Y,\mathscr{O}_Y):d_{\mathscr{A}}a=0\}
	=
	K,
	\]
	since $Y$ is projective and geometrically connected. Evaluation at $x$ is a $K$-linear tensor functor. It is exact because short exact sequences of
	vector bundles are locally split.
	If $a$ is a nonzero morphism, its image is a
	nonzero vector bundle on the connected curve $Y$,
	and hence has positive rank at $x$.
	Thus $\operatorname{im}\omega_x(a)
	\cong
	\omega_x(\operatorname{im}a)\neq0.$
	Therefore $\omega_x$ is faithful and is a fibre functor.
\end{proof}

\begin{remark}[Role of transitivity]
	\label{rem:transitivity-tannakian}
	On a smooth connected projective curve $Y$ over an
	algebraically closed field of characteristic zero,
	transitivity is also necessary for the full category
	of vector bundles with integrable Lie algebroid
	connections to be abelian. Indeed, let $\mathscr{A}$ be a Lie algebroid whose anchor fails to be surjective at a closed point
	$y\in Y$. Every vector field in the image of the anchor
	vanishes at $y$, so it preserves
	$\mathscr{I}_y=\mathscr{O}_Y(-y)$.
	The canonical connection on $\mathscr{O}_Y$
	therefore restricts to an integrable connection
	on $\mathscr{I}_y$, making $\mathscr{I}_y\hookrightarrow\mathscr{O}_Y$
	horizontal. This inclusion is both a monomorphism and an
	epimorphism in the vector-bundle category:
	any morphism from $\mathscr{O}_Y$ to a vector
	bundle that vanishes on $\mathscr{I}_y$ vanishes
	on a dense open subset and hence everywhere.
	Since the inclusion is not an isomorphism, the
	category is not abelian.
	Its fibre map at $y$ is zero, so evaluation at
	$y$ is not faithful. This obstruction concerns the full category and
	does not exclude Tannakian subcategories defined
	by additional conditions. Transitivity is not needed for tensor operations,
	fpqc descent, or the reconstruction equivalence
	of Lemma \ref{lem:tensor-reconstruction}.
\end{remark}

\subsection{The Dolbeault side}
For an $S$-scheme $T$, let $\mathcal{C}^{\mathrm{sh}}_{\mathrm{Dol}}(T)
:=
\text{Higgs}^{\mathrm{sh}}_{\mathscr{L}_T}(X_T/T)$ be the full subcategory of
$\mathcal{C}_{\mathrm{Dol}}(T)$ consisting of Higgs
bundles whose geometric fibres are semistable of
degree zero, together with the zero object.
Semistability is tested against Higgs-invariant
coherent subsheaves. The categories of families are exact tensor
categories. The neutral Tannakian assertion concerns
their geometric-fibre counterparts and requires
no transitivity hypothesis.

\begin{lemma}[Tensor operations in families]
	\label{lem:higgs-tensor}
	For every $S$-scheme $T$, the category
	$\mathcal{C}^{\mathrm{sh}}_{\mathrm{Dol}}(T)$
	is closed under tensor products, duals, finite
	direct sums, and extensions.
	With the exact structure inherited from the
	underlying vector bundles, it is a rigid exact
	tensor category with unit $(\mathscr{O}_{X_T},0)$. Every morphism $T'\to T$ induces an exact tensor pullback functor $\mathcal{C}^{\mathrm{sh}}_{\mathrm{Dol}}(T)
	\longrightarrow
	\mathcal{C}^{\mathrm{sh}}_{\mathrm{Dol}}(T').$
\end{lemma}

\begin{proof}
	The stability assertions are checked on geometric
	fibres. Tensor products preserve semistability of degree
	zero by Corollary \ref{cor:tensor-semiharmonic-relative}.
	For duals, a saturated invariant subbundle of
	positive degree in $E^\vee$ would give an invariant
	quotient of $E$ of negative degree, contradicting
	semistability.
	Extensions preserve semistability of degree zero:
	intersect an invariant subsheaf with the first
	term and take its image in the last term.
	Both have nonpositive degree, so the subsheaf does
	as well. Direct sums are a special case. For local sections $u$ of $L_T$, the tensor and dual Higgs fields are
	\[
	\theta_{E\otimes F}(u)
	=
	\theta_E(u)\otimes\operatorname{Id}_F
	+ \operatorname{Id}_E\otimes\theta_F(u),
	\qquad
	\theta_{E^\vee}(u)=-\theta_E(u)^t.
	\]
	They are integrable, and the usual symmetry,
	evaluation, and coevaluation maps are Higgs
	morphisms. Extension-closure gives the inherited exact
	structure. Tensoring preserves its short exact sequences,
	and these remain exact after arbitrary pullback,
	since they are locally split.
	The geometric-fibre semistability condition is
	also preserved by pullback.
\end{proof}

\begin{theorem}\label{thm:tannakian-Dol}
	Let $f\colon X\to S$ be a smooth projective family
	of geometrically connected curves over a
	$\mathbb{C}$-scheme, and let $\mathscr{L}$ be a
	relative Lie algebroid with locally free underlying
	sheaf. For every geometric point
	$\bar{s}\colon\operatorname{Spec}k\to S$, the category $\mathcal{H}_{\bar{s}}
	:=
	\text{Higgs}^{\mathrm{sh}}_{\mathscr{L}_{\bar{s}}}
	(X_{\bar{s}}/k)$
	is neutral Tannakian over $k$.
	For every $x\in X_{\bar{s}}(k)$, evaluation $\omega_x\colon\mathcal{H}_{\bar{s}}
	\longrightarrow\operatorname{Vect}_k,
	\ \
	(E,\theta)\longmapsto E_x,$
	is a fibre functor. 
\end{theorem}

\begin{proof}
	Put $Y=X_{\bar{s}}$ and
	$V=\Omega^1_{\mathscr{L}_{\bar{s}}}$.
	Let $a\colon(E,\theta)\longrightarrow(F,\eta)$
	be a morphism, and write $K=\ker a$ and $I=\operatorname{im}a$.
	Both are invariant and torsion-free, hence locally free on $Y$. Thus
	\[
	0\longrightarrow K\longrightarrow E
	\longrightarrow I\longrightarrow0
	\]
	is an exact sequence of vector bundles. Let $\overline{I}$ be the saturation of $I$ in $F$. The map $
	\overline{I}\longrightarrow
	(F/\overline{I})\otimes V$ induced by $\eta$ vanishes generically and has
	torsion-free target. It therefore vanishes everywhere, so
	$\overline{I}$ is invariant.
	Semistability gives $\deg K\leq0$ and
	$\deg\overline{I}\leq0$.
	Since $\deg E=0$, $0\leq\deg I\leq\deg\overline{I}\leq0.$
	Hence $K$ and $I$ have degree zero, and
	$\overline{I}/I$ has length zero.
	Thus $I$ is saturated and $Q:=F/I$ is locally free
	of degree zero. The induced Higgs bundles $K$ and $I$ are
	semistable as degree-zero invariant subsheaves.
	For an invariant subsheaf $J\subseteq Q$, its
	inverse image $\widetilde{J}\subseteq F$ is
	invariant, and
	\[
	\deg J
	=
	\deg\widetilde{J}-\deg I
	=
	\deg\widetilde{J}\leq0.
	\]
	Thus $Q$ is also semistable.
	The sheaf-theoretic kernels and cokernels therefore
	belong to $\mathcal{H}_{\bar{s}}$, and the
	coimage--image isomorphism respects the Higgs
	fields.
	Consequently, $\mathcal{H}_{\bar{s}}$ is abelian,
	every morphism is strict, and exact sequences are
	exact on the underlying vector bundles. By Lemma \ref{lem:higgs-tensor}, the category is rigid. Its unit is $(\mathscr{O}_Y,0)$, with $\operatorname{End}(\mathbb{1})
	=
	H^0(Y,\mathscr{O}_Y)=k.$
	Evaluation at $x$ is an exact $k$-linear tensor
	functor. If $a$ is nonzero, its image is a nonzero vector
	bundle on the connected curve $Y$, so $
	\operatorname{im}\omega_x(a)
	\cong
	\omega_x(\operatorname{im}a)\neq0.$ Thus $\omega_x$ is faithful and is a fibre functor.
\end{proof}

\begin{definition}\label{def:semiharmonic-principal}
	A principal $\mathscr{L}_T$-twisted $G$-Higgs bundle
	$(E_G,\theta_G)$ on $X_T$ is of
	\emph{semiharmonic type relative to $T$} if every
	associated Higgs bundle $E_G(W)$ belongs to
	$\mathcal{C}^{\mathrm{sh}}_{\mathrm{Dol}}(T)$. It is of \emph{harmonic type relative to $T$} if every associated Higgs bundle has polystable
	degree-zero geometric fibres.
\end{definition}

\begin{lemma}[Faithful-representation criterion]
	\label{lem:semiharmonic-faithful}
	Let $G$ be connected reductive over $\mathbb{C}$.
	A principal $\mathscr{L}_T$-twisted $G$-Higgs bundle
	is of semiharmonic type relative to $T$ if and
	only if $E_G(W)$ has semistable degree-zero
	geometric fibres for one faithful representation
	$W$ of $G$.
	The criterion is therefore independent of the
	chosen faithful representation.
\end{lemma}

\begin{proof}
	Necessity is immediate. Conversely, suppose that $E_G(W)$ satisfies the
	stated condition. Since $W$ is faithful, every representation $U$
	is a subquotient of a finite direct sum
	$P(W,W^\vee)$ of mixed tensor powers.
	Complete reducibility then gives an equivariant
	decomposition $P(W,W^\vee)\cong U\oplus U';$
	see \cite[Propositions~2.20(b) and~2.21(b)]
	{Deligne-Milne-Ogus-Shih-1982}. The induced decomposition
	\[
	P\bigl(E_G(W),E_G(W)^\vee\bigr)
	\cong E_G(U)\oplus E_G(U')
	\]
	respects the Higgs fields.
	By Lemma \ref{lem:higgs-tensor}, the left-hand side
	has semistable degree-zero geometric fibres.
	On each fibre, both summands have nonpositive
	degree, and their degrees sum to zero.
	Hence both have degree zero and are semistable.
	Since $U$ was arbitrary, the principal Higgs
	bundle is of semiharmonic type.
\end{proof}

\begin{corollary}[Principal reconstruction]
	\label{cor:tannakian-principal-Dol}
	For every $S$-scheme $T$, the associated-bundle
	construction induces an equivalence between:
	\begin{enumerate}[(i)]
		\item
		the groupoid of principal
		$\mathscr{L}_T$-twisted $G$-Higgs bundles on $X_T$
		of semiharmonic type relative to $T$;
		
		\item
		the groupoid of exact tensor functors $
		\operatorname{Rep}(G)\longrightarrow
		\mathcal{C}^{\mathrm{sh}}_{\mathrm{Dol}}(T),$
		with tensor natural isomorphisms.
	\end{enumerate}
	These equivalences commute with arbitrary base
	change and identify the corresponding fpqc stacks
	over $S$.
\end{corollary}

\begin{proof}
	Restrict the equivalence of
	Lemma \ref{lem:tensor-reconstruction}.
	By definition, a principal Higgs bundle is of
	semiharmonic type precisely when its associated
	tensor functor takes values in
	$\mathcal{C}^{\mathrm{sh}}_{\mathrm{Dol}}(T)$.
	The inherited exact structure makes this
	restriction exact. Base-change compatibility follows from the
	reconstruction lemma.
	Principal bundles and Higgs fields admit fpqc
	descent, and semiharmonic type is fpqc-local on
	the base. Thus the equivalences identify the corresponding
	fpqc stacks.
\end{proof}

\begin{remark}\label{rem:relative-tannakian}
	For $\mathscr L=T_{X/S}$,
	Corollary \ref{cor:tannakian-principal-Dol} recovers
	\cite[Lemma~9.2]{Simpson-94}.
	The exact tensor categories of families must be
	distinguished from the neutral Tannakian
	categories on geometric fibres.
	In general, neither
	$\mathcal{C}_{\mathrm{dR}}(T)$ nor
	$\mathcal{C}^{\mathrm{sh}}_{\mathrm{Dol}}(T)$
	is abelian. For example, on a constant family
	$Y\times\mathbb{A}^1\to\mathbb{A}^1$, with $Y$
	a smooth connected projective complex curve,
	multiplication by the base coordinate gives a
	morphism of the relative tensor unit $
	\mathscr{O}_{Y\times\mathbb{A}^1}
	\xrightarrow{\ t\ }
	\mathscr{O}_{Y\times\mathbb{A}^1}.$
	In both categories this is a monomorphism and
	an epimorphism, but not an isomorphism:
	multiplication by $t$ is injective on every
	vector bundle. Moreover, in both categories, $
	\operatorname{End}(\mathbb{1})
	=
	\Gamma(T,\mathscr{O}_T),$
	because $
	(f_T)_*\mathscr{O}_{X_T}=\mathscr{O}_T$ after arbitrary base change
	\cite[\href{https://stacks.math.columbia.edu/tag/0E0L}{Tag 0E0L}]{Stacks-project-2026}.A section $\sigma\colon T\to X_T$ gives exact tensor evaluation functors to $\operatorname{Vect}(T)$.
	On geometric fibres, these are fibre functors
	for the semiharmonic Higgs category by
	 Theorem \ref{thm:tannakian-Dol}, and for the de Rham
	category under the transitivity hypothesis of
	 Theorem \ref{thm:tannakian-dR}.
	No section is needed for reconstruction.
\end{remark}

\begin{remark}\label{rem:relative-integrability}
	Integrability is imposed on $X_T$, not merely
	on its geometric fibres.
	Over a nonreduced base, fibrewise vanishing of
	curvature or of $\theta\wedge\theta$ need not
	imply vanishing on the family. The geometric-fibre categories in
	Theorem \ref{thm:tannakian-dR} and Theorem \ref{thm:tannakian-Dol}
	contain all objects of the specified type on
	the geometric curve. No extension to a neighbourhood in $S$ is assumed.
\end{remark}

\section{Relative moduli of $\mathscr L$-twisted principal $G$-objects}
\label{sec:principal-moduli}
Let $f\colon X\to S$ be a smooth projective family
of geometrically connected curves, where $S$ is
of finite type over $\mathbb C$.
Fix an $f$-relatively ample line bundle
$\mathscr O_X(1)$ and a section
$\sigma\colon S\to X$.
Let $\mathscr L=(L,[\cdot,\cdot]_L,\delta)$
be a relative Lie algebroid on $X/S$, with $L$
locally free. In the de Rham case, we assume that the anchor $\delta\colon L\longrightarrow T_{X/S}$
is surjective. No transitivity assumption is imposed in the
Dolbeault case, where the construction depends
only on the underlying locally free sheaf $L$.
Let $G$ be a connected reductive algebraic group
over $\mathbb C$, viewed as a constant group
scheme over each base. Following the strategy of
\cite[\S~9]{Simpson-95}, we use a faithful
representation of $G$ to construct the framed
principal moduli scheme as a closed reduction
locus in a vector-bundle representation scheme.
The unframed moduli space is then obtained by
a relative good quotient.
The reduction condition is imposed on the family
itself, including over nonreduced bases, before
taking the quotient.
Closedness of the reduction locus and existence
of the quotient are proved below.

\subsection{Recalled results and notation}
\label{subsec:principal-relative-setup}
For an $S$-scheme $T$, write
\[
X_T:=X\times_S T,
\qquad
f_T\colon X_T\to T,
\qquad
\sigma_T\colon T\to X_T,
\]
and let $\mathscr L_T$ denote the induced relative
Lie algebroid.
For $\bullet\in\{\mathrm{dR},\mathrm{Dol}\}$, let
$\Lambda^\bullet_{\mathscr L/S}$ be the corresponding
sheaf of differential operators introduced above.
All statements indexed by $\bullet$ are understood
under the hypotheses for that case. A $\Lambda^{\mathrm{dR}}_{\mathscr L_T/T}$-object is a vector bundle with an integrable relative
$\mathscr L_T$-connection. A $\Lambda^{\mathrm{Dol}}_{\mathscr L_T/T}$-object
is a vector bundle with an integrable $\mathscr L_T$-twisted Higgs field.
Integrability is required on $X_T$, not merely on its geometric fibres. Let $\mathcal T^\bullet(T)$ denote the category
of these objects whose geometric fibres are
semistable of degree zero, including the zero
object. It carries the exact structure inherited from the underlying vector bundles.

On each connected component of $S$, set $
\mathcal P_0(m)
:=
\chi\bigl(
X_{\bar s},\mathscr O_{X_{\bar s}}(m)
\bigr),$ where $\bar s$ is any geometric point of that
component. This polynomial is independent of $\bar s$.
A rank-$r$, degree-zero bundle on a fibre has
Hilbert polynomial $r\mathcal P_0$.
On curve fibres, slope and Gieseker semistability
coincide; both are tested against invariant
coherent subsheaves.

We use the following results established above.
\begin{enumerate}[(i)]
	\item
	For every geometric point
	$\bar s\colon\Spec k\to S$, the category $
	\mathcal T^\bullet_{\bar s}
	:=\mathcal T^\bullet(\Spec k)$
	is neutral Tannakian, with fibre functor given
	by evaluation at $\sigma(\bar s)$. Its sheaf-theoretic kernels and cokernels are locally free and belong to the category.
	
	\item
	For every $r>0$, the functor of rank-$r$ objects
	in $\mathcal T^\bullet(T)$ equipped with a frame $
	\beta\colon
	\sigma_T^*E\xrightarrow{\sim}
	\mathscr O_T^{\oplus r}$
	is represented by a quasi-projective $S$-scheme $
	\mathcal R^\bullet_{\mathscr L}
	(X/S,\sigma,r)$
	carrying a universal framed object.
	Its formation commutes with arbitrary base
	change. The natural $\operatorname{GL}_{r,S}$-action
	admits a relative GIT linearisation for which
	every point is semistable.
	
	\item
	For every $S$-scheme $T$, the associated-bundle
	construction identifies the groupoid of principal
	$\mathscr L_T$-objects whose associated vector
	objects lie in $\mathcal T^\bullet(T)$ with the
	groupoid of exact tensor functors $
	\operatorname{Rep}(G)
	\longrightarrow
	\mathcal T^\bullet(T),$
	with tensor natural isomorphisms.
	These equivalences commute with arbitrary base change.
\end{enumerate}
\noindent
Assertion~(i) follows in the Dolbeault case from
Lemma \ref{lem:higgs-tensor} and Theorem \ref{thm:tannakian-Dol},
without any condition on the anchor.
In the de Rham case, the kernel, image, and
saturation argument gives abelianness of the
degree-zero semistable subcategory.
The dual-quotient argument and the tensor-product
theorem proved above give closure under duals
and tensor products.
Evaluation is exact and faithful by
Theorem \ref{thm:tannakian-dR}, using transitivity.

\noindent
Assertion~(ii) follows from
Theorem \ref{thm:moduli-vb-dR} and Proposition \ref{prop:twisted-Higgs-HK},
with Hilbert polynomial $r\mathcal P_0$.
Indeed, on every geometric curve fibre,
Jordan--H\"older factors are torsion-free and
hence locally free, so $\mathrm{LF}(\sigma)$
is automatic.

\noindent
Assertion~(iii) is the restriction of
Lemma \ref{lem:tensor-reconstruction} to the indicated
exact tensor subcategories; see also
Corollary \ref{cor:tannakian-principal-Dol}.
It applies over arbitrary, possibly nonreduced,
test schemes. Under this equivalence, a principal frame along
$\sigma_T$ corresponds to a tensor isomorphism $
\sigma_T^*\circ F
\xrightarrow{\sim}
\omega_G\otimes_{\mathbb C}\mathscr O_T,$
where $F$ is the associated tensor functor and
$\omega_G\colon\operatorname{Rep}(G)\to\text{Vect}_{\mathbb C}$ is the forgetful functor.

\subsection{The relative moduli functors}
\label{subsec:functors}
Retain the notation and the case-dependent hypotheses
of \S\ref{subsec:principal-relative-setup}.
In particular, transitivity is assumed only in the
de Rham case.
For a principal $G$--bundle $\mc F$ on $X_T$ and a
finite-dimensional representation $W$ of $G$, write $\mc F(W):=\mc F\times^G W.$

\begin{definition}
	\label{def:functors}
	Fix $\bullet\in\{\dR,\Dol\}$.
	A \emph{principal $G$--object} on $X_T/T$ is a pair
	$(\mc F,\xi)$ consisting of a principal $G$--bundle
	$\mc F$ on $X_T$ and an integrable relative
	$\mathscr L_T$--connection if $\bullet=\dR$, or an
	integrable $\mathscr L_T$--twisted Higgs field if
	$\bullet=\Dol$. Integrability is imposed on $X_T$. The pair is \emph{admissible} if, for every geometric
	point $\bar t\to T$ and every finite-dimensional
	representation $W$ of $G$, the associated
	$\Lambda^\bullet$--object $\mc F_{\bar t}(W)$ is
	semistable of degree zero.
	Equivalently, it is semistable with Hilbert polynomial $P_{\mc F_{\bar t}(W)}
	=
	(\dim W)\mc P_0.$ Define contravariant set-valued functors on
	$\mathrm{Sch}/S$ by
	\[
	\underline{\mc M}^{\bullet}_{\mathscr L}
	(X/S,G)(T)
	:=
	\left\{
	\text{admissible principal $G$--objects on $X_T/T$}
	\right\}\big/\cong
	\]
	and
	\[
	\underline{\mc M}^{\bullet}_{\mathscr L}
	(X/S,\sigma,G)(T)
	:=
	\left\{
	(\mc F,\xi,b)
	\ \middle|\
	\begin{array}{l}
		(\mc F,\xi)\text{ is admissible},\\
		b:T\to\sigma_T^*\mc F\text{ is a section}
	\end{array}
	\right\}\big/\cong.
	\]
	The section $b$ is a \emph{frame}; it trivialises the
	principal $G$--bundle $\sigma_T^*\mc F$. Isomorphisms preserve the connection or Higgs field and, in the framed case, the chosen frame.
	The maps induced by a morphism $T'\to T$ are given
	by pullback along $X_{T'}\to X_T$.
\end{definition}

\begin{remark}
	\label{rem:functor-conditions}
	\begin{enumerate}[(i)]
		\item
		The degree-zero condition for all associated bundles
		is equivalent to $\deg\mc F_{\bar t}(\chi)=0$ for every character
		$\chi:G\to\mathbf G_m,$
		since $\deg\mc F_{\bar t}(W)
		=
		\deg\mc F_{\bar t}(\det W).$
		On a smooth projective curve, this is also equivalent
		to vanishing of the positive-degree rational
		cohomological Chern classes of every associated
		bundle.
		
		\item
		Admissibility can be tested on one faithful
		representation $V$ of $G$.
		Indeed, every representation $W$ is a direct summand
		of a finite direct sum of mixed tensor powers of
		$V$ and $V^\vee$.
		Tensor products and duals preserve semistability
		of degree zero. Each invariant summand therefore has nonpositive
		degree, and the degrees of complementary summands
		sum to zero. Thus every summand is semistable of degree zero.
		The criterion is independent of the faithful
		representation.
		
		\item
		Admissibility implies fibrewise semistability of the
		induced adjoint object $(\ad(\mc F), \ad(\xi)).$
		Adjoint semistability alone does not impose the
		character-degree condition. We use the associated-object definition of
		admissibility throughout.
	\end{enumerate}
\end{remark}

\subsection{Closed conditions and invariant quotients}
\label{subsec:principal-auxiliary}
We first recall the relative vanishing-locus
construction used in \cite[proof of Theorem~3.8]{Simpson-94}.
\begin{lemma}[Relative vanishing locus]
	\label{lem:vanishing-locus}
	Let $T$ be a noetherian $S$--scheme, and let
	$\mc E,\mc F$ be coherent sheaves on $X_T$, with
	$\mc E$ flat over $T$.
	For every morphism
	$\varphi:\mc F\to\mc E$, there is a unique closed
	subscheme $Z(\varphi)\subseteq T$ such that, for
	every morphism $u:T'\to T$,
	\[
	u_X^*\varphi=0
	\quad\Longleftrightarrow\quad
	u\text{ factors through }Z(\varphi),
	\]
	where $u_X:X_{T'}\to X_T$ is the induced morphism.
	The construction commutes with arbitrary base change.
\end{lemma}

The next lemma extends the invariant-quotient
constructions of
\cite[Lemmas~9.3 and~9.9]{Simpson-95}
for ordinary Higgs bundles and integrable
connections, respectively.

\begin{lemma}[Invariant quotients]
	\label{lem:Invariant quotients}
	Let $T$ be an $S$--scheme, and let
	$E\in\mc T^\bullet(T)$ have rank $r$.
	For every $0\leq k\leq r$, there is a projective
	$T$--scheme $\mc N(E,k)$ representing isomorphism
	classes of finitely presented quotients $E_{T'}\twoheadrightarrow F,
	\ \ T'\to T,$ such that $F$ is flat over $T'$, has Hilbert polynomial
	$k\mc P_0$ on geometric fibres, and has invariant
	kernel. Every such quotient and its kernel are locally free,
	of ranks $k$ and $r-k$, respectively, and belong to
	$\mc T^\bullet(T')$. Restriction along the section induces a closed
	immersion $\mc N(E,k)\hookrightarrow
	\Grass_T(\sigma_T^*E,k),$ where the Grassmannian parametrises locally free
	rank-$k$ quotients. The construction and this closed immersion commute
	with arbitrary base change.
\end{lemma}

\begin{proof}
	The case $r=0$ is immediate.
	Assume first that $T$ is noetherian.
	Set $Q:=\Quot_{X_T/T}(E,k\mc P_0),$
	with universal sequence $ 0\longrightarrow K^{\rm univ}
	\xrightarrow{\iota}E_Q
	\xrightarrow{q}F^{\rm univ}
	\longrightarrow0.$ Write $D$ for the connection or Higgs field on $E_Q$.
	The composite
	\[
	\beta:=
	(q\otimes\operatorname{id}_{L_Q^\vee})
	\circ D\circ\iota:
	K^{\rm univ}
	\longrightarrow
	F^{\rm univ}\otimes L_Q^\vee
	\]
	is $\mathscr O_{X_Q}$--linear.
	In the connection case, the Leibniz term vanishes
	because $q\iota=0$.
	Thus $\beta=0$ precisely when the kernel is invariant. Its target is flat over $Q$, so Lemma \ref{lem:vanishing-locus} gives a closed subscheme $\mc N(E,k):=Z(\beta)\subseteq Q$
	representing invariant quotients.
	The induced structure on the quotient is integrable,
	since that on $E_Q$ is integrable.
	This proves projectivity and representability,
	including over arbitrary test schemes over $T$. Consider a geometric quotient
	\[
	0\longrightarrow K\longrightarrow E_{\bar t}
	\longrightarrow F\longrightarrow0.
	\]
	Its Hilbert polynomial gives $\rk F=k,\ \ \deg F=0,\ \ \deg K=0.$
	where the degree of a coherent sheaf includes the
	length of its torsion.
	Let $\overline K$ be the saturation of $K$ in
	$E_{\bar t}$.
	The map $\overline K\longrightarrow
	(E_{\bar t}/\overline K)\otimes L_{\bar t}^\vee$
	induced by $D$ is $\mathscr O$--linear, vanishes
	generically, and has torsion-free target.
	Hence $\overline K$ is invariant. Semistability gives
	\[
	0\geq\deg\overline K
	=
	\deg K+\operatorname{length}(\overline K/K)
	=
	\operatorname{length}(\overline K/K)\geq0,
	\]
	with the same inequality immediate when
	$\overline K=0$ or $\overline K=E_{\bar t}$.
	Thus $K$ is saturated and $F$ is locally free.
	Both $K$ and $F$ are semistable of degree zero,
	by the subsheaf and quotient criteria. The universal quotient is finitely presented, flat over $\mc N(E,k)$, and locally free on every
	geometric fibre. Since $X_{\mc N(E,k)}\to\mc N(E,k)$ is flat, the
	fibrewise flatness criterion shows that this quotient
	is flat on $X_{\mc N(E,k)}$.
	Finite presentation then makes it locally free.
	Its kernel is locally free as well.
	Restriction along the section therefore defines
	\[
	j:\mc N(E,k)\longrightarrow
	\Grass_T(\sigma_T^*E,k).
	\]
	
	We show that $j$ is a closed immersion.
	Suppose two geometric quotients, with kernels
	$K_1$ and $K_2$, have the same restriction along
	the section.
	The morphism
	\[
	K_1\longrightarrow E_{\bar t}
	\longrightarrow E_{\bar t}/K_2
	\]
	vanishes at $\sigma(\bar t)$.
	Its source and target belong to
	$\mc T^\bullet_{\bar t}$, so faithfulness of
	evaluation makes it zero.
	Interchanging the kernels gives $K_1=K_2$.
	This holds over every algebraically closed
	extension, so $j$ is universally injective. Let $n$ be a geometric point represented by a quotient $E_{\bar t}\twoheadrightarrow F$ with
	kernel $K$. Linearising the invariance condition gives $T_{\mc N(E,k)/T,n}
	\cong
	\operatorname{Hom}_{\mc T^\bullet_{\bar t}}(K,F).$
	Indeed, the tangent space to the Quot scheme is
	$\operatorname{Hom}_{\mathscr O_{X_{\bar t}}}(K,F)$,
	and the linearised condition requires the
	homomorphism to preserve the connection or
	Higgs field.
	Under this identification, the differential of
	$j$ is evaluation:
	\[
	\operatorname{Hom}_{\mc T^\bullet_{\bar t}}(K,F)
	\longrightarrow
	\operatorname{Hom}_{\kappa(\bar t)}
	\bigl(\sigma_{\bar t}^*K,\sigma_{\bar t}^*F\bigr).
	\]
	It is injective by faithfulness.
	Thus $j$ has zero relative tangent spaces and,
	being locally of finite type, is unramified.
	It is also proper, since its source is proper
	over $T$ and its target is separated over $T$.
	Hence $j$ is a closed immersion
	\cite[\href{https://stacks.math.columbia.edu/tag/04XV}{Tag 04XV}]{Stacks-project-2026}. For arbitrary $T$, work locally so that
	$\sigma_T^*E$ is trivial. he resulting framed object is pulled back from
	the universal object on the noetherian scheme
	$\mathcal R^\bullet_{\mathscr L}(X/S,\sigma,r)$.
	Pulling back the construction above gives
	$\mc N(E,k)$ and its closed immersion into the
	Grassmannian. The local constructions glue by their universal
	property, which also proves compatibility with
	arbitrary base change.
\end{proof}

\begin{corollary}[Relative rigidity]
	\label{cor:principal-relative-rigidity}
	For every $S$--scheme $T$ and
	$A,B\in\mc T^\bullet(T)$, restriction along the section
	is injective:
	\[
	\operatorname{Hom}_{\Lambda^\bullet}(A,B)
	\longrightarrow
	\operatorname{Hom}_{\mathscr O_T}
	(\sigma_T^*A,\sigma_T^*B).
	\]
	Moreover, a morphism in $\mc T^\bullet(T)$ is an
	isomorphism if its restriction along $\sigma_T$
	is an isomorphism.
\end{corollary}

\begin{proof}
	Work locally on $T$ so that the ranks are constant.
	Let $a:A\to B$ vanish along $\sigma_T$.
	Its graph and the graph of the zero morphism are
	invariant subbundles of $A\oplus B$, both with
	quotient isomorphic to $B$.
	They define the same quotient along the section.
	The closed immersion of
	Lemma \ref{lem:Invariant quotients} therefore identifies
	the two graphs, so $a=0$. If $\sigma_T^*a$ is an isomorphism, the fibre functors show that $a$ is an isomorphism on every geometric
	fibre. It is consequently an isomorphism of vector bundles
	on $X_T$.
\end{proof}

\subsection{Monodromy and reduction of structure group}
\label{subsec:principal-monodromy}
Let $G\hookrightarrow H$ be a closed embedding of
connected reductive algebraic groups over $\mathbb C$.
For a principal $H$--object $(\mc F_H,\xi_H)$ on
$X_T/T$ and $W\in\text{Rep}(H)$, the associated bundle
$\mc F_H(W)$ carries its induced connection or
Higgs field.

\begin{definition}
	\label{def:relative-monodromy}
	Let $(\mc F_H,\xi_H,b_H)$ be an admissible framed
	principal $H$--object on $X_T/T$.
	Its \emph{monodromy is contained in $G$} if, for every
	$W\in\text{Rep}(H)$ and every $G$--invariant subspace
	$U\subseteq W$, there is an invariant subbundle $F(U,W)\subseteq\mc F_H(W)$
	whose quotient belongs to $\mc T^\bullet(T)$ and whose
	restriction along $\sigma_T$, under the frame $b_H$,
	is
	\[
	\sigma_T^*F(U,W)
	=
	U\otimes_{\mathbb C}\mathscr O_T
	\subseteq
	W\otimes_{\mathbb C}\mathscr O_T.
	\]
\end{definition}

Every such $F(U,W)$ belongs to $\mc T^\bullet(T)$:
on each geometric fibre, it is an invariant
degree-zero subbundle of a semistable degree-zero
object.
It is also unique.
Indeed, for two choices $F_1$ and $F_2$, the composite
\[
F_1\longrightarrow\mc F_H(W)
\longrightarrow\mc F_H(W)/F_2
\]
vanishes along the section, hence vanishes by
Corollary \ref{cor:principal-relative-rigidity}.
Thus $F_1\subseteq F_2$, and symmetry gives equality.
The monodromy condition is preserved by arbitrary
base change, and uniqueness shows that the
subbundles $F(U,W)$ commute with pullback.

\begin{lemma}[Relative reduction of structure group]
	\label{lem:monodromy}
	Extension of structure group induces a bijection,
	functorial in the $S$--scheme $T$, between:
	\begin{enumerate}[(i)]
		\item
		isomorphism classes of admissible framed principal
		$G$--objects on $X_T/T$;
		
		\item
		isomorphism classes of admissible framed principal
		$H$--objects on $X_T/T$ whose monodromy is contained
		in $G$.
	\end{enumerate}
\end{lemma}

\begin{proof}
	Let $(\mc F_G,\xi_G,b_G)$ be an admissible framed
	principal $G$--object.
	Its extension to $H$ is admissible, since every
	$H$--representation restricts to a $G$--representation.
	For $W\in\text{Rep}(H)$ and a $G$--invariant subspace
	$U\subseteq W$, the associated sequence
	\[
	0\longrightarrow\mc F_G(U)
	\longrightarrow\mc F_H(W)
	\longrightarrow\mc F_G(W/U)
	\longrightarrow0
	\]
	is compatible with the additional structures and
	the framings.
	It supplies the required invariant subbundle.
	
	Conversely, let $(\mc F_H,\xi_H,b_H)$ have monodromy
	contained in $G$.
	Consider the category $\text{Rep}(G,H)$ of pairs $(U,W)$,
	where $W\in\text{Rep}(H)$ and $U\subseteq W$ is
	$G$--invariant, with all $G$--equivariant maps
	$U\to U'$ as morphisms.
	The functor
	\[
	\text{Rep}(G,H)\longrightarrow\text{Rep}(G),
	\qquad (U,W)\longmapsto U,
	\]
	is fully faithful and essentially surjective.
	For the latter assertion, choose a faithful
	$H$--representation $V$.
	Its restriction to $G$ is faithful, so every
	$G$--representation embeds in a finite direct sum
	of mixed tensor powers of $V$ and $V^\vee$.
	These tensor constructions are representations
	of $H$. Assign to $(U,W)$ the unique subbundle $F(U,W)$.
	To construct the induced morphisms, let
	$u:U\to U'$ be $G$--equivariant.
	Its graph is a $G$--invariant subspace $\Gamma_u\subseteq W\oplus W'.$
	The projections from $F(\Gamma_u,W\oplus W')$ factor through $F(U,W)$
	and $F(U',W')$. Indeed, their composites with the corresponding
	quotients vanish along the section and hence vanish
	by Corollary \ref{cor:principal-relative-rigidity}.
	The first projection
	\[
	F(\Gamma_u,W\oplus W')
	\longrightarrow F(U,W)
	\]
	is an isomorphism, since its restriction along the
	section is the isomorphism $\Gamma_u\to U$.
	The second projection therefore defines a morphism $u^\sharp:F(U,W)\longrightarrow F(U',W')$
	restricting to $u\otimes\operatorname{id}_{\mathscr O_T}$.
	Relative rigidity makes this lift unique.
	It follows that the construction respects identities,
	composition, and linear combinations.
	Transporting along the equivalence
	$\operatorname{Rep}(G,H)\simeq\operatorname{Rep}(G)$ gives a
	$\mathbb C$--linear functor on $\operatorname{Rep}(G)$. For tensor compatibility, the composite
	\[
	F(U,W)\otimes F(U',W')
	\longrightarrow
	\mc F_H(W\otimes W')
	\longrightarrow
	\frac{\mc F_H(W\otimes W')}
	{F(U\otimes U',W\otimes W')}
	\]
	vanishes along the section.
	Its source and target belong to
	$\mc T^\bullet(T)$, so relative rigidity makes it
	zero.
	The resulting morphism
	\[
	F(U,W)\otimes F(U',W')
	\longrightarrow
	F(U\otimes U',W\otimes W')
	\]
	restricts to the identity on $U\otimes U'$ and is
	therefore an isomorphism.
	The unit is $F(\mathbb C,\mathbb C)=\mathscr O_{X_T}$
	with its unit structure.
	Associativity, symmetry, and unit compatibility
	follow from relative rigidity, since the corresponding
	diagrams commute along the section.
	We have constructed a tensor functor
	\[
	\Phi_G:\operatorname{Rep}(G)\longrightarrow\mc T^\bullet(T).
	\]
	It is exact because $\operatorname{Rep}(G)$ is semisimple.
	The frame $b_H$ gives a tensor trivialisation of
	$\sigma_T^*\Phi_G$. By Lemma \ref{lem:tensor-reconstruction}, this functor
	and its tensor trivialisation determine an admissible framed principal $G$--object $(\mc F_G,\xi_G,b_G).$
	For every $W\in\operatorname{Rep}(H)$, uniqueness gives $
	F(W,W)=\mc F_H(W).$ The same uniqueness identifies the induced maps
	for $H$--equivariant morphisms with the original
	associated-bundle maps. Thus restricting $\Phi_G$ along
	$\operatorname{Rep}(H)\to\operatorname{Rep}(G)$ recovers the original framed
	$H$--object, including its connection or Higgs field.
	Its extension of structure group is therefore
	canonically isomorphic to
	$(\mc F_H,\xi_H,b_H)$. If the $H$--object originally comes from a framed
	$G$--object, uniqueness identifies every $F(U,W)$
	with the corresponding associated subbundle
	$\mc F_G(U)$. The reconstruction hence recovers that $G$--object.
	An isomorphism of framed $H$--objects preserves these
	unique subbundles and induces an isomorphism of the
	reconstructed $G$--objects.
	The two constructions are therefore inverse on
	isomorphism classes.
	
	Finally, the subbundles, their induced morphisms,
	and their tensor identifications commute with
	pullback by uniqueness.
	The bijection is consequently functorial in $T$.
\end{proof}

\subsection{The relative representation scheme}
\label{subsec:principal-representation}

\begin{theorem}
	\label{thm:rep-space}
	Under the case-dependent hypotheses of
	\S\ref{subsec:principal-relative-setup}, let $G$ be
	a connected reductive algebraic group over
	$\mathbb C$.
	For $\bullet\in\{\dR,\Dol\}$, the framed functor $\underline{\mc M}^{\bullet}_{\mathscr L}
	(X/S,\sigma,G)$
	is represented by a quasi-projective $S$--scheme $\mc R^{\bullet}_{\mathscr L}(X/S,\sigma,G).$ Its formation commutes with arbitrary base change:
	for every morphism $S'\to S$, there is a canonical
	isomorphism
	\[
	\mc R^{\bullet}_{\mathscr L}(X/S,\sigma,G)
	\times_S S'
	\simeq
	\mc R^{\bullet}_{\mathscr L_{S'}}
	(X_{S'}/S',\sigma_{S'},G).
	\]
	Every closed embedding $G\hookrightarrow H$ of
	connected reductive algebraic groups induces a
	closed immersion $\mc R^{\bullet}_{\mathscr L}(X/S,\sigma,G)
	\hookrightarrow
	\mc R^{\bullet}_{\mathscr L}(X/S,\sigma,H).$
	In particular, a faithful representation
	$G\hookrightarrow\GL(V)$, with $\dim V=n$, identifies
	the source with a closed subscheme of the framed
	rank-$n$ vector-bundle representation scheme $\mc R^{\bullet}_{\mathscr L}(X/S,\sigma,n).$
\end{theorem}

\begin{proof}
	For $G=\GL(V)$, the frame-bundle construction
	identifies principal objects with rank-$n$
	vector-bundle objects, compatibly with their
	connections or Higgs fields and their framings.
	By Remark \ref{rem:functor-conditions}, admissibility
	can be tested on the standard representation.
	The required scheme is therefore supplied by the
	relative vector-bundle construction. Suppose that $G\hookrightarrow H$ is a closed embedding and that the framed $H$--scheme $R_H:=
	\mc R^{\bullet}_{\mathscr L}(X/S,\sigma,H)$
	has been constructed.
	Let
	$(\mc F_H^{\rm univ},\xi_H^{\rm univ},b_H^{\rm univ})$
	be its universal framed object.
	For $W\in\Rep(H)$ and a $G$--invariant subspace
	$U\subseteq W$, set
	\[
	E_W:=\mc F_H^{\rm univ}(W),
	\qquad
	k:=\dim(W/U).
	\]
	The universal frame identifies
	$\sigma_{R_H}^*E_W$ with
	$W\otimes_{\mathbb C}\mathscr O_{R_H}$.
	Thus $W\twoheadrightarrow W/U$ defines a section
	\[
	\gamma_{W,U}:R_H\longrightarrow
	\Grass_{R_H}(\sigma_{R_H}^*E_W,k).
	\]
	Using Lemma \ref{lem:Invariant quotients}, form the
	closed subscheme
	\[
	\mathbf C(W,U)
	:=
	R_H
	\times_{\gamma_{W,U},\,
		\Grass_{R_H}(\sigma_{R_H}^*E_W,k)}
	\mc N(E_W,k).
	\]
	A morphism $T\to R_H$ factors through
	$\mathbf C(W,U)$ precisely when $(E_W)_T$ has an
	invariant subbundle restricting to
	$U\otimes_{\mathbb C}\mathscr O_T$ along the section,
	with quotient in $\mc T^\bullet(T)$.
	Such a subbundle is unique. Choose a set of representatives for the pairs
	$(W,U)$ and define
	\[
	R_{G\mid H}:=
	\bigcap_{(W,U)}\mathbf C(W,U)
	\subseteq R_H
	\]
	scheme-theoretically.
	Its ideal sheaf is the sum of the ideal sheaves
	of the $\mathbf C(W,U)$.
	Since $R_H$ is noetherian, finitely many pairs
	already define this intersection. For every $S$--scheme $T$, factorisation through $R_{G\mid H}$ is equivalent to the monodromy
	condition of \ref{def:relative-monodromy}.
	Hence Lemma \ref{lem:monodromy} gives a natural
	identification
	\[
	\Hom_S(T,R_{G\mid H})
	\simeq
	\underline{\mc M}^{\bullet}_{\mathscr L}
	(X/S,\sigma,G)(T).
	\]
	Thus $R_{G\mid H}$ represents the framed
	$G$--functor, and its inclusion into $R_H$
	represents extension of structure group. Taking $H=\GL(V)$ constructs the required scheme for $G$ and proves quasi-projectivity.
	Applying the same construction to any closed
	embedding $G\hookrightarrow H$ gives the asserted
	closed immersion, by uniqueness of representing
	schemes. Finally, for an arbitrary morphism $S'\to S$,
	the framed functor for the base-changed family
	is the restriction of the original functor to
	$S'$--schemes. Indeed, the objects, their framings, and the
	geometric-fibre admissibility condition are
	unchanged. The base-change isomorphism follows from the
	representing property.
\end{proof}

\begin{remark}
	\label{rem:representation-scheme-intrinsic}
	The representation scheme is independent of the
	faithful representation used in its construction,
	up to the unique isomorphism compatible with its
	representing property.
	It represents the intrinsic functor of admissible
	framed principal $G$--objects.
	No passage to an underlying reduced scheme is made.
\end{remark}

\subsection{The relative good quotient}
\label{subsec:principal-good-quotient}
Retain the case-dependent hypotheses of
\S\ref{subsec:principal-relative-setup}.
Fix a faithful representation
$G\hookrightarrow\GL(V)$, with $\dim V=n$, and write
\[
R_n:=\mc R^\bullet_{\mathscr L}(X/S,\sigma,n),
\qquad
R_G:=\mc R^\bullet_{\mathscr L}(X/S,\sigma,G).
\]
The group $\GL(V)$ acts on $R_n$ by change of
framing. The closed immersion $R_G\hookrightarrow R_n$
of Theorem \ref{thm:rep-space} is equivariant for the
restricted $G$--action. Let $\mathscr A$ be the relatively ample
$\GL(V)$--linearised line bundle supplied by the
vector-bundle construction, for which every point
of $R_n$ is relatively semistable.
Its restriction $\mathscr A_G:=\mathscr A|_{R_G}$ is relatively ample and $G$--linearised.

\begin{lemma}
	\label{lem:all-semistable}
	Every point of $R_G$ is relatively $G$--semistable
	with respect to $\mathscr A_G$.
\end{lemma}

\begin{proof}
	Work locally on $S$, with $S$ affine.
	For $x\in R_G$, semistability in $R_n$ gives
	$\nu>0$ and a $\GL(V)$--invariant section $s\in\Gamma(R_n,\mathscr A^{\otimes\nu})$ such that $s(x)\neq0$ and $(R_n)_s$ is affine
	over $S$. The restriction $s_G:=s|_{R_G}$ is $G$--invariant,
	does not vanish at $x$, and has nonvanishing locus $(R_G)_{s_G}=R_G\cap(R_n)_s.$
	This is closed in $(R_n)_s$, hence affine over $S$.
	Thus $x$ is relatively $G$--semistable.
\end{proof}

Relative GIT therefore gives a good quotient
\begin{equation}
	\label{eq:good-quotient}
	q_G:R_G\longrightarrow
	\mc M^\bullet_{\mathscr L}(X/S,G)
	:=
	R_G\sslash_S G,
\end{equation}
whose target is quasi-projective over $S$. We record the scalar-extension property used below. Let $B$ be a $\mathbb C$--algebra, let $A$ be a
$G$--equivariant $B$--algebra, and let $B'$ be a
$B$--algebra, with trivial $G$--action on $B$ and
$B'$.
Then $ A^G\otimes_B B'
\xrightarrow{\sim}
(A\otimes_B B')^G.$ Indeed, take a presentation of $B'$ by free
$B$--modules.
Tensoring with $A$ and taking invariants proves
the assertion, since $G$ is linearly reductive
and invariants commute with direct sums.
No flatness assumption on $B'$ is needed.
Applied on affine quotient charts, this shows
that $q_G$ commutes with arbitrary base change
on $S$.

\begin{theorem}
	\label{thm:moduli-principal}
	Under the case-dependent hypotheses of
	\S\ref{subsec:principal-relative-setup}, let $G$ be
	a connected reductive algebraic group over
	$\mathbb C$.
	For $\bullet\in\{\dR,\Dol\}$, the quasi-projective
	$S$--scheme $\mc M^\bullet_{\mathscr L}(X/S,G)$
	universally corepresents $\underline{\mc M}^\bullet_{\mathscr L}(X/S,G).$ Its formation commutes with arbitrary base change:
	for every morphism $S'\to S$, there is a canonical
	isomorphism
	\[
	\mc M^\bullet_{\mathscr L}(X/S,G)\times_S S'
	\simeq
	\mc M^\bullet_{\mathscr L_{S'}}
	(X_{S'}/S',G).
	\]
	In particular, its geometric fibre at $\bar s\to S$
	is canonically $\mc M^\bullet_{\mathscr L_{\bar s}}
	(X_{\bar s},G).$
\end{theorem}

\begin{proof}
	Put $
	M_G:=\mc M^\bullet_{\mathscr L}(X/S,G),
	\ \
	\mathcal F_G:=
	\underline{\mc M}^\bullet_{\mathscr L}(X/S,G),$
	and write $h_Y=\Hom_S(-,Y)$ for an $S$--scheme $Y$. We first define $\eta:\mathcal F_G\longrightarrow h_{M_G}.$
	Let $(\mc F,\xi)$ be an admissible object over $T$.
	The principal bundle $\sigma_T^*\mc F$ admits
	sections $b_i$ on an \'etale cover $\{T_i\to T\}$.
	These frames give morphisms $u_i:T_i\longrightarrow R_G.$
	On overlaps, the frames differ by the action of
	$G$. The morphisms $q_Gu_i$ therefore agree and descend
	to a morphism $T\to M_G$.
	This construction is independent of the frames,
	depends only on the isomorphism class of the
	object, and commutes with pullback.
	It defines $\eta$. Let $\alpha:\mathcal F_G\to h_Y$ be a natural
	transformation. Applying it to the universal object over $R_G$,
	after forgetting the frame, gives a morphism $a:R_G\longrightarrow Y.$
	Naturality and independence of the frame make
	$a$ $G$--invariant. The categorical property of $q_G$ gives a unique
	morphism
	\[
	\bar a:M_G\longrightarrow Y,
	\qquad a=\bar a q_G.
	\]
	For an arbitrary object over $T$, choose local
	frames as above.
	Naturality gives
	\[
	\alpha_T(\mc F,\xi)|_{T_i}
	=
	\bar a q_Gu_i
	=
	\bigl(\bar a\eta_T(\mc F,\xi)\bigr)|_{T_i}.
	\]
	Equality descends to $T$.
	Thus $\alpha=h_{\bar a}\circ\eta$, and uniqueness
	follows by evaluation on the universal object.
	Hence $M_G$ corepresents $\mathcal F_G$. To prove universal corepresentability, consider
	\[
	\pi:[R_G/G]\longrightarrow M_G.
	\]
	On an affine quotient chart, this is
	$[\Spec A/G]\to\Spec A^G$.
	Exactness of invariants shows that $\pi$ is a
	good moduli space morphism.
	By
	\cite[Proposition~4.7(i) and Theorem~4.16(vi)]
	{Alper-GMS}, its arbitrary base changes are
	good moduli space morphisms and are universal
	for morphisms to schemes. Explicitly, let $M'\to M_G$ be any morphism and set $R':=R_G\times_{M_G}M'.$ Then $[R'/G]\to M'$ is a good moduli space morphism, so $R'\to M'$ is categorical for $G$--invariant
	morphisms to schemes. The pulled-back functor $\mathcal F_G\times_{h_{M_G}}h_{M'}$ is described, after choosing local frames, by
	morphisms to $R'$. The preceding descent and factorisation argument,
	with $R'\to M'$ in place of $q_G$, therefore shows
	that $M'$ corepresents this functor.
	Thus $\eta$ is universally corepresenting. Finally, Theorem  \ref{thm:rep-space} gives arbitrary base-change compatibility for $R_G$, and the
	scalar-extension calculation above gives the
	same compatibility for its quotient.
	This proves the asserted base-change isomorphism
	and its geometric-fibre specialisation.
\end{proof}

\begin{remark}
	\label{rem:principal-quotient-points}
	The scheme $M_G$ corepresents the functor of
	isomorphism classes; it need not represent it. On a geometric fibre, two admissible principal $G$--objects determine the same moduli point
	precisely when the closures of their framed
	$G$--orbits intersect in $(R_G)_{\bar s}$.
	This condition is independent of the frames,
	and each equivalence class contains a unique
	closed orbit. In particular, the good quotient need not carry
	a universal unframed principal $G$--object.
	Both $R_G$ and $M_G$ retain their full scheme
	structures; neither is assumed to be reduced.
\end{remark}


\section*{Acknowledgments}
The first author thanks Sumit Roy for helpful discussions and comments
and acknowledges support from the National Board of Higher Mathematics
(NBHM) through its Doctoral Research Fellowship Program.

\end{document}